\documentclass[12pt]{amsart}
\usepackage{amssymb}
\usepackage{lscape}
\usepackage{tikz}
\usepackage{fullpage}
\usepackage{lmodern}
\usepackage[T1]{fontenc}
\usepackage{paralist}
\usepackage{color}
\usepackage{tabmac}
\usepackage{amscd}
\usepackage[all]{xy}
\usepackage[normalem]{ulem}
\usepackage[colorlinks,pagebackref=true,pdftex]{hyperref}
\usepackage[alphabetic,abbrev,backrefs,lite]{amsrefs}

\makeatletter

\@addtoreset{equation}{section}
\def\theequation{\thesection.\@arabic \c@equation}

\def\theenumi{\@roman\c@enumi}

\def\@citecolor{blue}
\def\@linkcolor{blue}
\def\@urlcolor{blue}

\makeatother

\newtheorem{lemma}[equation]{Lemma}
\newtheorem{theorem}[equation]{Theorem}

\newtheorem{prop}[equation]{Proposition}

\newtheorem{claim*}{Claim}
\newtheorem{thm}[equation]{Theorem}

\theoremstyle{definition}
\newtheorem{rmk}[equation]{Remark}
\newenvironment{remark}[1][]{\begin{rmk}[#1] \pushQED{\qed}}{\popQED \end{rmk}}
\newtheorem{eg}[equation]{Example}
\newenvironment{example}[1][]{\begin{eg}[#1] \pushQED{\qed}}{\popQED \end{eg}}
\newtheorem{defn}[equation]{Definition}
\newenvironment{definition}[1][]{\begin{defn}[#1]\pushQED{\qed}}{\popQED \end{defn}}
\newtheorem{con}[equation]{Construction}
\newenvironment{construction}[1][]{\begin{con}[#1]\pushQED{\qed}}{\popQED \end{con}}

\def\<{\langle}
\def\>{\rangle}

\newcommand{\ch}{\operatorname{char}}
\newcommand{\codim}{\operatorname{codim}}

\newcommand{\coker}{\operatorname{coker}}

\newcommand{\Ext}{\operatorname{Ext}} 

\newcommand{\Hom}{\operatorname{Hom}}
\newcommand{\sheafHom}{\mathcal{H}om}

\newcommand{\Supp}{\operatorname{Supp}}

\newcommand{\Hs}{\mathrm{H}}

\newcommand{\EE}{\mathcal{E}}
\newcommand{\FF}{\mathcal{F}}

\newcommand{\GG}{\mathcal{G}}
\newcommand{\KK}{\mathcal{K}}
\newcommand{\cL}{\mathcal{L}}

\newcommand{\cO}{{\mathcal O}}
\newcommand{\PP}{\mathbb{P}}
\newcommand{\cQ}{\mathcal{Q}}
\newcommand{\Sc}{\mathbf{S}}
\newcommand{\ZZ}{\mathbb{Z}}
\newcommand{\NN}{\mathbb{N}}
\newcommand{\RR}{\mathbb{R}}
\newcommand{\QQ}{\mathbb{Q}}

\newcommand{\GL}{{\bf GL}}
\DeclareMathOperator{\Kos}{Kos}
\newcommand{\Sym}{\operatorname{Sym}}

\newcommand{\defi}[1]{{\upshape\bfseries #1}}

\DeclareMathOperator{\tate}{T}
\newcommand{\minus}{\ensuremath{\!\smallsetminus\!}}

\newcommand{\excise}[1]{}

\newsavebox{\explicit}
\title{Poset structures in Boij--S\"oderberg theory}

\author{Christine Berkesch}
\address{Institut Mittag-Leffler \\ Aurav\"agen 17 \\ SE-182 60
  Djursholm, Sweden 
  \hfill \quad \linebreak
\indent Department of Mathematics \\ Stockholm University \\
SE-106 91 Stockholm, Sweden}
\email{cberkesc@math.su.se}

\author{Daniel Erman}
\address{Department of Mathematics \\ Stanford University \\
Stanford, CA 94305}
\email{derman@math.stanford.edu}

\author{Manoj Kummini}
\address{Department of Mathematics \\ Purdue University \\
West Lafayette, IN 47907}
\curraddr{Chennai Mathematical Institute, Siruseri, Tamilnadu 603103.
India}
\email{mkummini@cmi.ac.in}

\author{Steven V Sam}
\address{Department of Mathematics \\ Massachusetts Institute of
Technology \\
Cambridge, MA 02139}
\email{ssam@math.mit.edu}

\thanks{ The first author was partially supported by NSF Grants DMS
  0901123 and OISE 0964985.  The second author was partially supported
  by an NDSEG fellowship and NSF Award No. 1003997.  The fourth author
  was supported by an NSF graduate research fellowship and an NDSEG
  fellowship.}

\date{}

\begin{document}

\begin{abstract}
%
  Boij--S\"oderberg theory is the study of two cones: the cone of Betti diagrams of 
  standard graded minimal free resolutions over a polynomial ring and
  the cone of cohomology tables of coherent sheaves over projective
  space.
We provide a new interpretation of these partial orders in terms of the
existence of nonzero homomorphisms, 
for both the general and equivariant constructions.
These results provide new insights into the families of modules and sheaves 
at the heart of Boij--S\"oderberg theory: 
Cohen--Macaulay modules with pure resolutions and supernatural sheaves. 
In addition, they suggest the naturality of these partial orders and
provide tools for extending Boij--S\"oderberg theory 
to other graded rings and projective varieties.
%
\end{abstract}
\maketitle
\vspace{-.5cm}
\section{Introduction}
\label{sec:intro}

Boij--S\"oderberg theory is the study of the 
cone of Betti diagrams over the standard graded polynomial ring 
$S=\Bbbk[x_1,\dots, x_n]$ and -- dually -- the 
cone of cohomology tables of coherent sheaves on $\PP^{n-1}_\Bbbk$, 
where $\Bbbk$ is a field.  The extremal rays of these
cones correspond to special modules and sheaves: 
Cohen--Macaulay modules with pure resolutions (Definition~\ref{def:pure:res}) 
and supernatural sheaves (Definition~\ref{def:supernatural}), respectively.  
Each set of extremal rays carries a partial order $\preceq$
(Definitions~\ref{defn:partial:deg} and \ref{defn:partial:root}) that
induces a simplicial decomposition of the corresponding cone.

Each partial order $\preceq$ is defined in terms of certain combinatorial
data associated to these special modules and sheaves. 
For a module with a pure resolution, this data is a degree sequence, and for
a supernatural sheaf, this data is a root sequence. 
Our main results reinterpret these partial orders $\preceq$ 
in terms of the existence of nonzero homomorphisms 
between Cohen--Macaulay modules with pure resolutions 
and between supernatural sheaves.

\begin{thm}\label{thm:poset:deg:main}
Let $\rho_d$ and $\rho_{d'}$ be extremal rays of the cone of Betti diagrams
for $S$ corresponding to Cohen--Macaulay modules with pure resolutions of
types $d$ and $d'$, respectively. 
Then $\rho_d \preceq \rho_{d'}$ if and only if there
exist Cohen--Macaulay modules $M$ and $M'$ with pure resolutions of types
$d$ and $d'$, respectively, with $\Hom_S(M',M)_{\leq 0}\ne 0$.
\end{thm}

\begin{thm}\label{thm:poset:root:main}
Let $\rho_f$ and $\rho_{f'}$ be extremal rays of the cone of cohomology
tables for $\PP^{n-1}$ corresponding to supernatural sheaves of types $f$
and $f'$, respectively. 
Then $\rho_f\preceq \rho_{f'}$ if and only if there exist
supernatural sheaves $\EE$ and $\EE'$ of types $f$ and $f'$, respectively,
with $\Hom_{\PP^{n-1}}(\EE',\EE)\ne 0$.
\end{thm}

Though the statements of these two theorems are quite parallel,
Theorem~\ref{thm:poset:deg:main} is far more subtle than
Theorem~\ref{thm:poset:root:main}.  Theorem~\ref{thm:poset:root:main}
follows nearly directly from the Eisenbud--Schreyer pushforward
construction of supernatural sheaves, but without modification, 
it is not clear how to compare the modules constructed 
in~\cite[\S5]{EiScConjOfBS07}.

We illustrate this via an example.  Let $n=3$, $d=(0,2,3,5)$,
$d'=(0,3,9,10)$, and $M$ and $M'$ be finite length modules with
pure resolutions of types $d$ and $d'$, as constructed
in~\cite[\S5]{EiScConjOfBS07}.  
We know of no method to produce a nonzero element of $\Hom(M,M')_{\leq 0}$, 
even in this specific case.
The difficulty here stems from differences in the constructions of $M$ and
$M'$: the module $M$ is constructed by pushing forward a complex of
projective dimension $5$ along $\PP^2\times (\PP^1)^2\to \PP^2$,
whereas $M'$ is constructed by pushing forward a complex of
projective dimension $10$ along $\PP^2\times \PP^2\times \PP^5\to\PP^2$.  
Thus, the construction of~\cite[\S5]{EiScConjOfBS07} does not
even suggest that Theorem~\ref{thm:poset:deg:main} ought to be true.

Our motivation for conjecturing the statement of 
Theorem~\ref{thm:poset:deg:main} -- and the first key idea behind its proof -- 
is based on a flexible version of the Eisenbud--Schreyer construction of 
pure resolutions. This is Construction~\ref{modif:es} below, 
and we show that the basic results of~\cite[\S5]{EiScConjOfBS07} can be 
adapted to this construction. This extension enables us to use a single 
projection map to simultaneously produce modules $N$ and $N'$ with pure 
resolutions of types $d$ and $d'$.  In the case under consideration, 
we construct both $N$ and $N'$ by pushing forward complexes 
of projective dimension $10$ along the projection map 
$\PP^2\times (\PP^1)^7\to \PP^2$.\footnote{We note that $M\ne N$ 
and $M'\ne N'$ in this example.}

We may then produce elements of $\Hom(N,N')_{\leq 0}$ by working with 
the complexes on the source $\PP^2\times (\PP^1)^7$ of the projection map.  
However, finding such a nonzero element poses a second technical challenge 
in the proof of Theorem~\ref{thm:poset:deg:main}.  This requires an explicit 
and somewhat delicate computation involving the pushforward of a morphism 
of complexes along the projection $\PP^2\times (\PP^1)^7\to \PP^2$. 
This computation is carried out in the proof of 
Theorem~\ref{thm:deg:after:reduction}, thus providing a new understanding 
of how certain modules with pure resolutions are related. 

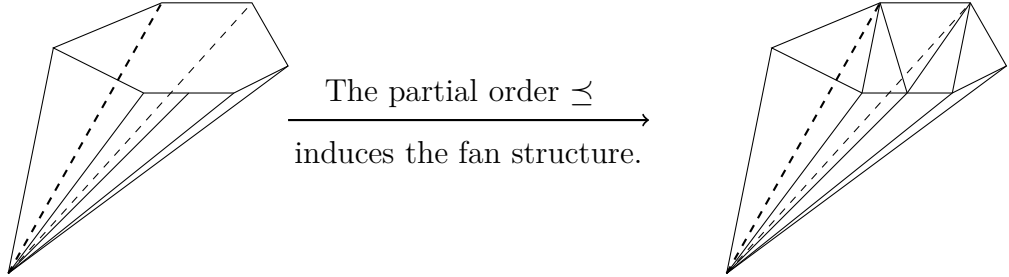
\begin{figure}
\begin{tikzpicture}[scale=1.2]
\draw[-](.5,.5)--(1,3);
\draw[-](.5,.5)--(2,2.5);
\draw[-](.5,.5)--(3,2.5);
\draw[-](.5,.5)--(3.6,2.8);
\draw[-](.5,.5)--(2.5,2.5);
\draw[dashed,-](.5,.5)--(3.2,3.5);
\draw[dashed,-,thick](.5,.5)--(2.2,3.5);
\draw[-](1,3)--(2,2.5)--(3,2.5)--(3.6,2.8)--(3.2,3.5)--(2.2,3.5)--cycle;
\draw[->,thick](3.6,2.2)--(7.6,2.2);
\draw (5.5,2.5) node { The partial order $\preceq$};
\draw (5.6,1.85) node { induces the fan structure.};
\end{tikzpicture}
\qquad
\begin{tikzpicture}[scale=1.2]
\draw[-](.5,.5)--(1,3);
\draw[-](.5,.5)--(2,2.5);
\draw[-](.5,.5)--(3,2.5);
\draw[-](.5,.5)--(3.6,2.8);
\draw[-](2,2.5)--(2.2,3.5);
\draw[-](3,2.5)--(3.2,3.5);
\draw[-](3.2,3.5)--(2.5,2.5);
\draw[-](2.2,3.5)--(2.5,2.5);
\draw[-](.5,.5)--(2.5,2.5);
\draw[dashed,-](.5,.5)--(3.2,3.5);
\draw[dashed,-,thick](.5,.5)--(2.2,3.5);
\draw[-](1,3)--(2,2.5)--(3,2.5)--(3.6,2.8)--(3.2,3.5)--(2.2,3.5)--cycle;
\end{tikzpicture}
\caption{
The partial order $\preceq$ on the extremal rays induces a simplicial
decomposition of the cone of Betti diagrams, where the 
simplices correspond to chains of extremal rays with respect to the partial order.
This simplicial decomposition is essential to many applications of
Boij--S\"oderberg theory.}
\end{figure}

Besides providing greater insight into the structure of modules with pure
resolutions and supernatural sheaves, our results have two further
implications.  First, the partial orders $\preceq$ are defined in
terms of the combinatorial data of degree sequences and root sequences 
(see Sections~\ref{sec:prelim deg} and~\ref{sec:prelim root}), and
depend on the total order of $\ZZ$; thus, they are only formally
related to $S$ and $\PP^{n-1}$.  However, our reinterpretations of 
$\preceq$ in terms of module- and sheaf-theoretic properties 
suggest the naturality not only of $\preceq$, but also of the
induced simplicial decompositions of both cones. In other words,
while there exist graded modules whose Betti diagrams can be written
as a positive sum of pure tables in several ways, Theorem~\ref{thm:poset:deg:main} 
suggests that the most natural of these decompositions is the
Boij--S\"oderberg decomposition produced
by~\cite{EiScConjOfBS07}*{Decomposition Algorithm}, and similarly for
Theorem~\ref{thm:poset:root:main} and cohomology tables.

A second implication involves the extension of 
Boij--S\"oderberg theory to more complicated projective
varieties or graded rings.  
For instance, the cone of free resolutions over a quadric hypersurface ring of 
$\Bbbk[x,y]$ is described in~\cite{bbeg}.  The extremal rays in this case correspond
to pure resolutions of finite or infinite length.  We could thus consider
a partial order defined in parallel to Boij--S\"oderberg's original definition (based
on the combinatorial data of a degree sequence), or, following our result, we could consider
a partial order defined in terms of nonzero homomorphisms.
These partial orders are different in this hypersurface case; only the second definition leads to
a decomposition algorithm for Betti diagrams.  See Example~\ref{ex:hypersurface} below for details.

For more general graded rings there even exist extremal rays that do
not correspond to pure resolutions.  (Similar statements hold for more
general projective varieties.)  There is thus no obvious extension of
Boij--S\"oderberg's original partial order to these cases.  By
contrast, the reinterpretations of $\preceq$ provided by
Theorems~\ref{thm:poset:deg:main} and ~\ref{thm:poset:root:main} are
readily applicable to arbitrary projective varieties and graded rings.
We discuss one such case in Example~\ref{ex:bigraded}.

Theorems~\ref{thm:poset:deg:main} and~\ref{thm:poset:root:main} hold
over an arbitrary field $\Bbbk$, and their proofs involve variants of
the constructions in~\cite{EiScConjOfBS07} for supernatural sheaves
and modules with pure resolutions.  When $\ch(\Bbbk)=0$, there also
exist equivariant constructions of supernatural vector
bundles~\cite{EiScConjOfBS07}*{Thm.~6.2} and of finite length modules
with pure resolutions~\cite{efw}*{Thm.~0.1}.  For these we prove the most
natural equivariant analogues of 
our main results. 

\begin{thm}\label{thm:equivariant:deg}
Let $V$ be an $n$-dimensional $\Bbbk$-vector space with 
$\ch(\Bbbk)=0$, and let $\rho_d$ and $\rho_{d'}$ be
the extremal rays of the cone of Betti diagrams for $S=\Sym(V)$
corresponding to finite length modules 
with pure resolutions of types $d$ and $d'$.  
Then $\rho_d\preceq \rho_{d'}$ if and only if there
exist finite length $\GL(V)$-equivariant 
modules $M$ and $M'$ 
with pure resolutions of types $d$ and $d'$, respectively, 
with $\Hom_{\GL(V)}(M',M)_{\le 0}\ne 0$.
\end{thm}

\begin{thm}\label{thm:equivariant:root}
Let $V$ be an $n$-dimensional $\Bbbk$-vector space with
$\ch(\Bbbk)=0$, and let $\rho_f$ and $\rho_{f'}$ be the extremal rays
of the cone of cohomology tables for $\PP^{n-1} = \PP(V)$
corresponding to supernatural vector bundles of types $f$ and $f'$.
Then $\rho_f\preceq \rho_{f'}$ if and only if there exist
$\GL(V)$-equivariant supernatural vector bundles $\EE$ and $\EE'$ of
types $f$ and $f'$, respectively, with $\Hom_{\GL(V)}(\EE',\EE)\ne 0$.
\end{thm}

The action of $\GL(V)$ has two orbits on the maximal ideals of $S$:
one consisting of the maximal ideal $(x_1, \dots, x_n)$ and the other
consisting of its complement. An equivariant Cohen--Macaulay module
therefore has only two options for its support, and hence either has finite
length or must be a free module. Thus the finite length hypothesis in
Theorem~\ref{thm:equivariant:deg} is the natural equivariant analogue
of the Cohen--Macaulay hypothesis in Theorem~\ref{thm:poset:deg:main}.

As above, the statement for pure resolutions is more subtle than the
corresponding statement for supernatural vector bundles. The modules
constructed in \cite[\S 3]{efw} do not have nonzero equivariant
homomorphisms between them, but the explicit combinatorics of the
representation theory involved suggests a minor modification which
does work. This also suggests how the maps should be defined in terms of
the explicit presentation of the modules; the remaining nontrivial step is to
show that these maps are in fact well-defined. The main obstacle is 
that such maps must be compatible with the actions of both the
general linear group and the symmetric algebra, and the interplay
between the two is delicate. This key issue in the proof of
Theorem~\ref{thm:equivariant:deg} is accomplished through a careful 
computation involving Pieri maps (combined with results from~\cite{sam}).

\subsection*{Outline}
In this paper, we first focus on the cone of Betti diagrams for $S$. 
In Section~\ref{sec:prelim deg}, we prove the reverse implications of 
Theorems~\ref{thm:poset:deg:main} and~\ref{thm:equivariant:deg}. 
We then construct nonzero morphisms 
between modules with pure resolutions. 
Sections~\ref{sec:construct deg} and~\ref{sec:equiv:deg}, 
respectively, address the forward directions of 
Theorems~\ref{thm:poset:deg:main} and~\ref{thm:equivariant:deg}.
We next address the cone of cohomology tables for $\PP^{n-1}$. 
In Section~\ref{sec:prelim root}, we prove the reverse
implications of Theorems~\ref{thm:poset:root:main}
and~\ref{thm:equivariant:root}.  We then turn to the construction of
nonzero morphisms between supernatural sheaves:
Sections~\ref{sec:construct root} and~\ref{sec:equiv:root},
respectively, address the forward directions of
Theorems~\ref{thm:poset:root:main} and~\ref{thm:equivariant:root}.
Finally, we provide in Section~\ref{sec:extensions} 
a brief discussion of how Theorem~\ref{thm:poset:deg:main} has been 
applied in the study of Boij--S\"oderberg theory over other graded rings.
We suggest the survey~\cite{ES:ICMsurvey}  
to the reader seeking additional background on Boij--S\"oderberg theory. 

\subsection*{Acknowledgements}
We would like to thank J.~Burke, D.~Eisenbud, C.~Gibbons, W. F.~Moore, F.-O.~Schreyer, B.~Ulrich, and J.~Weyman
for helpful discussions. Significant parts of this work were completed
at the Pan American Scientific Institute Summer School on
``Commutative Algebra and its Connections to Geometry'' in Olinda,
Brazil, and when the second author visited Purdue University; we thank
both of these institutions for their hospitality.  The computer
algebra system \texttt{Macaulay2} \cite{M2} provided valuable
assistance in studying examples. 

\section{The poset of degree sequences}
\label{sec:prelim deg}

Let $M$ be a finitely generated graded $S$-module. 
The \defi{$(i,j)$th graded Betti number} of $M$, 
denoted $\beta_{i,j}(M)$, is $\dim_\Bbbk \mathrm{Tor}_i^S(\Bbbk, M)_j$. 
The \defi{Betti diagram} of $M$ is a table, with rows indexed by $\ZZ$ 
and columns by $0, \dots, n$, 
such that the entry in column $i$ and row $j$ is $\beta_{i,i+j}(M)$.
A sequence $d=(d_0, \dots, d_n)\in \left( \ZZ\cup \{\infty\} \right)^{n+1}$
is called a \defi{degree sequence} for $S$ if $d_i>d_{i-1}$ for all $i$ 
(with the convention that $\infty>\infty$). 
The \defi{length} of $d$, denoted $\ell(d)$, 
is the largest integer $t$ such that $d_t$ is finite.

\begin{definition}\label{def:pure:res}
A graded $S$-module $M$ is said to have a \defi{pure resolution of type
$d$} if a
minimal free resolution of $M$ has the form
\[
0\leftarrow M \leftarrow S(-d_0)^{\beta_{0,d_0}} \leftarrow S(-d_1)^{\beta_{1,d_1}} \leftarrow
\dots
\leftarrow 
S(-d_{\ell(d)})^{\beta_{\ell(d),d_{\ell(d)}}} \leftarrow 0.
\qedhere
\]
\end{definition}

For every degree sequence $d$, there exists a 
Cohen--Macaulay module with a
pure resolution of type $d$~\cite{EiScConjOfBS07}*{Theorem~0.1} (see also~\cite{boij-sod1}*{Conjecture~2.4}, \cite{efw}*{Theorem~0.1}). 
The Betti diagram of any finitely generated $S$-module can be written as a positive
rational combination of the Betti diagrams of Cohen--Macaulay modules with
pure resolutions (see ~\cite {EiScConjOfBS07}*{Theorem~0.2} and~\cite{BoijSoderbergNonCM08}*{Theorem~2}). 
The \defi{cone of Betti diagrams} for $S$ 
is the convex cone inside $\bigoplus_{j \in \ZZ}\QQ^{n+1}$ 
generated by the Betti diagrams of all finitely generated $S$-modules.
Each degree sequence $d$ corresponds to a unique extremal ray of this cone, 
which we denote by $\rho_d$, and  
every extremal ray is of the form $\rho_d$ for some degree sequence $d$.

\begin{definition}\label{defn:partial:deg}
For two degree sequences $d$ and $d'$, we say that $d\preceq d'$  and that
$\rho_d\preceq \rho_{d'}$ if $d_i\leq d_i'$ for all $i$. 
\end{definition}
This partial order induces a simplicial fan structure on the cone of Betti diagrams, 
where simplices correspond to chains of degree sequences under the partial
order $\preceq$.
We now show that the existence of a nonzero homomorphism between 
two modules with pure resolutions implies the comparability of their 
corresponding degree sequences. 
This result provides the reverse implications for 
Theorems~\ref{thm:poset:deg:main} and~\ref{thm:equivariant:deg}. 

\begin{prop}\label{prop:half:poset:deg}
Let $M$ and $M'$ be graded Cohen--Macaulay $S$-modules 
with pure resolutions of types $d$ and $d'$, respectively. 
If $\Hom(M',M)_{\leq 0}\ne 0$, then $d\preceq d'$.
\end{prop}
\begin{proof}
Write $\ell' = \ell(d')$ and $\ell = \ell(d)$.  If $\ell' > \ell$, then
$\codim M' > \codim M$, and, by~\cite{BrHe:CM}*{Propositions~1.2.3
and~1.2.1}, $\Hom(M', M) = 0$.

Therefore we may assume that $\ell' \leq \ell$. 
By hypothesis, we may fix a nonzero homomorphism 
$\phi \in \Hom(M',M)_t$ for some $t\leq 0$. 
Let $F_\bullet$ and $F'_\bullet$ be minimal graded free resolutions of $M$
and $M'$, respectively, and let $\left\{ \phi_i\colon  F'_i \to F_i
\right\}_{i \geq 0}$ be the comparison maps in a lifting of $\phi$. 
Suppose by way of contradiction  that 
there is a $j$ such that $d'_j < d_j$. 
Since $d'_j < d_j$, we see that $\phi_j = 0$.  Hence,
each $\phi_i$ such that $j \leq i \leq \ell'$ can be made
zero by some homotopy equivalence.
Write $(-)^\vee = \Hom_S(-, S(-n))$.
Since $M$ and $M'$ are Cohen--Macaulay, 
we note that $(F_\bullet)^\vee$ and $(F'_\bullet)^\vee$ 
are minimal graded free resolutions of 
$\Ext^\ell_S(M,S(-n))$ and $\Ext^{\ell'}_S(M', S(-n))$. 
Further, the maps $\{\phi_i^\vee\}_{i \geq 0}$ define an element of 
$\Ext^{\ell-\ell'}_S\left(\Ext_S^\ell(M, S(-n)), \Ext_S^{\ell'}(M', S(-n))\right)$. 
In fact, if we write 
$N = \coker \left((F_{\ell'-1})^\vee \longrightarrow (F_{\ell'})^\vee\right)$, 
then $(\phi_{\ell'})^\vee \colon N \longrightarrow \Ext_S^{\ell'}(M', S(-n)))$ 
is the zero homomorphism.
Hence $\phi_i^\vee = 0$ for all $0 \leq i \leq \ell'$, and therefore $\phi = 0$.
\end{proof}

Proposition~\ref{prop:half:poset:deg} is untrue if we do not 
assume that $M'$ is Cohen--Macaulay. 
For example, consider $S = \Bbbk[x,y]$, $M = S/\langle x^2\rangle $, and
$M' = S \oplus \Bbbk$. 
We used the hypothesis that $M'$ is Cohen--Macaulay 
to have that $\codim M' = \ell(d')$ and that
$\Hom_S(F'_\bullet, S(-n))$ is a resolution. 

\section{Construction of morphisms between modules with pure resolutions}
\label{sec:construct deg}

In Theorem~\ref{thm:poset:deg:main} we must, necessarily, consider more
than $\Hom(M',M)_0$. 
For instance, if $n=2, d=(0,1,2)$, and $d'=(1,2,3)$, then any $M$ and
$M'$ with pure resolutions of types $d$ and $d'$ will be isomorphic to
$\Bbbk^m$ and $\Bbbk(-1)^{m'}$, respectively, for some integers $m,m'$.  
In this case, $\Hom(M',M)_0=0$, whereas $\Hom(M',M)_{-1}\ne 0$.

However, it is possible to reduce to the consideration of $\Hom(M',M)_0$.
To do this, let $t := \min \{d'_i - d_i \mid d'_i\neq \infty \}$. 
By replacing $d'$ by $d' - (t,\dots,t)$, 
the forward direction of Theorem~\ref{thm:poset:deg:main} 
is an immediate corollary of the following result. 

\begin{thm}
\label{thm:deg:after:reduction}
Let $d \preceq d'$ be degree sequences for $S$ with 
$d_j = d'_j$ for some $0 \leq j \leq \ell(d')$. 
Then there exist finitely generated graded Cohen--Macaulay
modules $M$ and $M'$ with pure resolutions of types $d$ and $d'$, respectively,
with $\Hom(M',M)_0\ne 0$.
\end{thm}

\begin{remark}
The homomorphism group in Theorems~\ref{thm:poset:deg:main} 
and~\ref{thm:deg:after:reduction} 
is nonzero only for specific choices of the modules $M$ and $M'$. 
For two degree sequences $d \preceq d'$, there exist many pairs of modules
$M$, $M'$ with pure resolutions of types $d$ and $d'$, respectively,  
such that $\Hom(M',M)_{\leq 0}=0$. 
For example, take $d=d'=(0,2,4)$, 
$M = S/\< x^2,y^2 \>$, and $M' = S/\< l_1^2,l_2^2 \>$ 
for general linear forms $l_1$ and $l_2$. 
As another example, consider $d = (0,3,6) \prec d'= (0,4,8)$.  When $M
= S/\< x^3,y^3 \>$ and $M' = S/\< f,g \>$ for general quartic
forms $f$ and $g$, we again have $\Hom(M',M)_{\leq 0}=0$.
\end{remark}

The proof of Theorem~\ref{thm:deg:after:reduction} is given at the end 
of this section and involves two main steps. 

\begin{enumerate}
\item\label{eq:deg:step:1}
Construct twisted Koszul complexes $\KK_\bullet$ and $\KK_{\bullet}'$ on
a product $\PP$ of projective spaces (including a copy of
$\PP^{n-1}$) and push them forward along the
projection $\pi\colon  \PP\to \PP^{n-1}$.  This yields pure 
resolutions $F_\bullet$ and $F_\bullet'$ of types $d$ and $d'$ that 
respectively resolve modules $M$ and $M'$. 
\item\label{eq:deg:step:2}
Show that there exists a morphism $h_\bullet\colon  \KK'_\bullet \to
\KK_\bullet$ such that the induced map $\nu_\bullet \colon  F_\bullet' \to
F_\bullet$ is not null-homotopic. 
This yields a nonzero element $\psi\in \Hom_S(M', M)_0$. 
\end{enumerate}

We achieve \eqref{eq:deg:step:1} by modifying the construction of pure
resolutions by Eisenbud and Schreyer \cite{EiScConjOfBS07}*{\S5}.  
We replace their use of $\prod_i \PP^{d_i - d_{i-1}}$ with a product of
copies of $\PP^1$.  This enables us to simultaneously 
construct pure resolutions of types 
$d$ and $d'$ and a nonzero map between the modules they resolve.  The
details of \eqref{eq:deg:step:1} are contained in
Construction~\ref{modif:es}.  For \eqref{eq:deg:step:2}, we 
apply Construction~\ref{modif:es} so as to produce the morphism $h_\bullet$.
Checking that the induced map $\nu_\bullet$ is not null-homotopic 
uses, in an essential way, the hypothesis that $d_j = d'_j$ for some $0
\leq j \leq \ell(d')$.  Example~\ref{ex:degseqs} demonstrates these
arguments. Write $\PP^{1 \times r}$ for the $r$-fold product of
$\PP^1$.

\begin{construction}
[Modification of the Eisenbud--Schreyer construction of pure resolutions]
\label{modif:es}
The objects involved in this construction of a pure resolution $F_\bullet$
of type $d$ will be denoted by $\Kos^d_\bullet$, $\KK_\bullet$, and $\cL$.
The corresponding objects for the pure resolution $F'_\bullet$ of type $d'$
are $\Kos^{d'}_\bullet$, $\KK'_\bullet$, and $\cL'$. Let 
\begin{equation}
\label{equation:rForKos}
r :=\max\{d_{\ell(d)}-d_0-\ell(d), d'_{\ell(d')}-d_0-\ell(d')\}
\end{equation}
and $\PP := \PP^{n-1}\times \PP^{1\times r}$. On $\PP$, fix the coordinates 
\[
\left([x_1:x_2:\dots:x_n], [y^{(1)}_0:y^{(1)}_1], \ldots,
[y^{(r)}_0:y^{(r)}_1]\right) 
\] 
and consider the multilinear forms 
\[
f_p := \sum_{i_0 + \cdots + i_r = p} 
\left(	x_{i_0}\cdot \prod_{j=1}^r y_{i_j}^{(j)} \right)
	\qquad \text{for } p = 1,2,\dots,n+r.
\]
(Note that $i_0 \in \{1, \ldots, n\}$ and $i_j \in \{0,1\}$ for all $1 \leq
j \leq r$.) We now define 
\begin{align*}
D & := \{d_0, d_0+1, \dots, d_0+\ell(d)+r\}, &
D' & := \{d_0, d_0+1, \dots, d_0+\ell(d')+r\}, \\
\delta & := (\delta_1< \dots< \delta_r) = D \minus d, &
\delta' & := (\delta'_1< \dots< \delta'_r) = D' \minus d', \\
a & := \delta - (d_0+1, \ldots, d_0+1), & 
a' & := \delta' - (d_0+1, \ldots, d_0+1),& \\
\cL & := \cO_\PP(-d_0,a), \qquad \qquad \qquad \qquad \text{and} & 
\cL' & := \cO_\PP(-d_0,a').
\end{align*}
(We view $\delta$ and $\delta'$ as ordered sequences.) 
Let $\Kos^{d}_\bullet$ be the Koszul complex on
$f_1, \dots, f_{\ell(d)+r}$, which is an acyclic complex of sheaves on
$\PP$ of length $\ell(d)+r$ (see
\cite{EiScConjOfBS07}*{Proposition~5.2}).  Let $\KK_\bullet
:= \Kos^d_\bullet \otimes \cL$.  Let $\pi\colon\PP \rightarrow
\PP^{n-1}$ denote the projection onto the first factor. By repeated
application of~\cite{EiScConjOfBS07}*{Proposition~5.3}, 
$\pi_*\KK_\bullet$ is an acyclic complex of sheaves on $\PP^{n-1}$ of
length $\ell(d)$ such that each term is a direct sum of line
bundles. Taking global sections of this complex in all twists yields
the pure resolution $F_\bullet$ of a graded $S$-module (that is
finitely generated and Cohen--Macaulay).  We
can write the free module $F_i$ explicitly as follows.  If $s=\max\{i \mid a_i-d_j+d_0 \leq -2\}$,
then we have
\[
F_j  = S(-d_j)^{\binom{\ell(d)+r}{d_j-d_0}}\otimes \left(
    \bigotimes_{i=1}^s \Hs^1(\PP^1, \cO(a_i-d_j+d_0)) \right) \otimes
  \left( \bigotimes_{i=s+1}^r \Hs^0(\PP^1, \cO(a_i-d_j+d_0)) \right) .
\]
Let $\Kos^{d'}_\bullet$ be the Koszul complex on 
$f_1, \dots, f_{\ell(d')+r}$ and $\KK_\bullet' := \Kos^{d'}_\bullet \otimes \cL'$, 
and define $F'_\bullet$ in a similar manner.
\end{construction}

The value of $r$ in \eqref{equation:rForKos} is the least integer such
that we are able to fit both the twists 
$-d_0$ and $\min\{-d_{\ell(d)}, -d'_{\ell(d')}\}$ 
in the $\PP^{n-1}$ coordinate of the bundles 
of the complexes $\KK_\bullet$ and $\KK'_\bullet$. The
choices of $a$ and $a'$, which ensure that $F_\bullet$ and
$F'_\bullet$ are pure of types $d$ and $d'$, are dictated by the
homological degrees in $\KK_\bullet$ and $\KK_\bullet'$ that need to
be eliminated in each projection away from a $\PP^1$ component of $\PP$. 
In Example~\ref{ex:degseqs}, these homological degrees are those with an
underlined $-1$ in Table~\ref{tab:propaedeuticTwists}. 
Observe that $a - a' \in \NN^r$ since $d\preceq d'$. 
Thus there is a nonzero map
$h_\bullet\colon\KK'_\bullet\to\KK_\bullet$ that is induced by a
polynomial of multidegree $(0,a-a')$.  In \eqref{eq:deg:step:2}, we
show that $\pi_*h_\bullet$ induces the desired nonzero map. 

The following extended example contains all of the main ideas
behind the proof of Theorem~\ref{thm:deg:after:reduction}.

\begin{example}
\label{ex:degseqs}
Consider $d=(0,2,4,5,6)$ and $d'=(1,2,4,7) = (1,2,4,7,\infty)$. 
Note that $d_2 = d'_2 = 4$, so that $d$ and $d'$ satisfy the hypotheses
of Theorem~\ref{thm:deg:after:reduction}. 
Here $r=4$ and $\PP = \PP^{3}\times \PP^{1\times 4}$. 
On $\PP$, we have the Koszul complexes 
$\Kos^d_\bullet = \Kos_\bullet(\cO_\PP; f_1, \dots, f_8)$ 
and $\Kos^{d'}_\bullet = \Kos_\bullet(\cO_\PP; f_1, \dots, f_7)$. 
There is a natural map $\Kos^{d'}_\bullet\to \Kos^{d}_\bullet$ 
induced by the inclusion 
$\langle f_1, \dots, f_7\rangle \subseteq \langle f_1, \dots, f_8\rangle$. 
Here we have  
\begin{align*}
\delta & =(1,3,7,8), & \delta'=&\ (0,3,5,6), \\
a&=(0,2,6,7),  & a'=&\ (-1,2,4,5), \\
\KK_\bullet&=\Kos^d_\bullet\otimes \cO_\PP(0,a), &\text{and}\quad 
\KK'_\bullet=&\ \Kos^{d'}_\bullet\otimes \cO_\PP(0,a').
\end{align*}
Table~\ref{tab:propaedeuticTwists} shows the twists 
in each homological degree of these complexes. 

\begin{table}[h]
\parbox{5cm}{%
\begin{tabular}{|c|c|}%
\multicolumn{2}{c}{$d = (0, 2, 4, 5, 6)$}\\%
\hline
$i$ & Twist in $\KK_i$\\\hline
0  &$(0,0,  2,  6,  7)$\\%
$-1$ &$(-1,\uuline{-1},  1,  5,  6)$ \\%
$-2$ &$(-2,-2,  0,  4,  5)$ \\%
$-3$ &$(-3,-3,  \uuline{-1},  3,  4)$ \\%
$-4$ &$(-4,-4,  -2,  2,  3)$ \\%
$-5$ &$(-5,-5,  -3,  1,  2)$ \\%
$-6$ &$(-6,-6,  -4,  0,  1)$ \\%
$-7$ &$(-7,-7,  -5,  \uuline{-1},  0)$\\%
$-8$ &$(-8,-8,  -6,  -2,  \uuline{-1})$ \\%
\hline%
\end{tabular}}
\parbox{5cm}{%
\begin{tabular}{|c|c|}%
\multicolumn{2}{c}{$d' = (1, 2, 4, 7)$} \\
\hline
$i$& Twist in $\KK'_i$\\\hline
$ 0$ & $(0,\uuline{-1}, 2, 4, 5)$\\
$-1$ & $(-1,-2, 1, 3, 4)$\\
$-2$ & $(-2,-3, 0, 2, 3)$\\
$-3$ & $(-3,-4, \uuline{-1}, 1, 2)$\\
$-4$ & $(-4,-5, -2, 0, 1)$\\
$-5$ & $(-5,-6, -3, \uuline{-1}, 0)$\\
$-6$ & $(-6,-7, -4, -2, \uuline{-1})$\\
$-7$ & $(-7,-8, -5, -3, -2)$\\
\hline
\end{tabular}
\phantom{ALIGNMENT line}
\vspace*{1.5mm}
}
\caption{Twists appearing in $\KK_\bullet$ and $\KK_\bullet'$ in Example~\ref{ex:degseqs}.}
\label{tab:propaedeuticTwists}
\end{table}

Let $h$ be a nonzero homogeneous polynomial on
$\PP$ of multidegree $(0,a-a')= (0,1,0,2,2)$. Then multiplication by
$h$ induces a nonzero map $h \colon \KK_0' \to \KK_0$. To write $h$, 
we use matrix multi-index notation for the monomials in
$\Bbbk[y_0^{(1)}, y_1^{(1)}, \dots, y_0^{(4)}, y_1^{(4)}]$, where the
$i$th column represents the multi-index of the
$y^{(i)}$-coordinates. With this convention, fix 
\[
h = 
\mathbf{y}^{\left(\begin{smallmatrix}1&0&2&2\\0&0&0&0
\end{smallmatrix}\right)} 
:= y^{(1)}_0\cdot \left( y^{(3)}_0\right)^2\cdot
\left(y^{(4)}_0\right)^{2}.
\]
Denote the induced map of complexes $\KK_\bullet'\to \KK_\bullet$ by
$h_\bullet$.
Taking the direct image of $h_\bullet$ along the natural projection
$\pi\colon \PP\to\PP^3$ and its global sections in all twists induces 
a map $\nu_\bullet\colon  F_\bullet'\to F_\bullet$.

We claim that $\nu_\bullet$ is not null-homotopic. This need not hold
for an arbitrary pair $d\preceq d'$, however it does hold for a pair
of degree sequences which satisfy the hypotheses of
Theorem~\ref{thm:deg:after:reduction}.  We use the fact that
$d_2=d_2'=4$, as this implies that $\nu_2\colon F_2' \to F_2$ is a
matrix of scalars.  Since $F_\bullet'$ and $F_\bullet$ are both
minimal free resolutions, it then follows that the map $\nu_2$ factors
through a null-homotopy only if $\nu_2$ is itself the zero map.  Thus
it is enough to show that $\nu_2\ne 0$.  For this, note that
\begin{align*}
F_2&=S(-4)^{\binom{8}{4}}\otimes \Hs^1(\PP^1, \cO(-4))\otimes \Hs^1(\PP^1, \cO(-2))\otimes \Hs^0(\PP^1, \cO(2))\otimes \Hs^0(\PP^1, \cO(3))\\
\text{and}\quad
F_2'&=S(-4)^{\binom{7}{4}}\otimes \Hs^1(\PP^1, \cO(-5))\otimes \Hs^1(\PP^1,
\cO(-2))\otimes \Hs^0(\PP^1, \cO(0))\otimes \Hs^0(\PP^1, \cO(1))
\end{align*}
and that $F_2$ and $F_2'$ have 
$\Hs^1$ terms in precisely the same positions, 
and similarly for the $\Hs^0$ terms. 
We may then use~\cite{explicit}*{Lemma 7.3} to compute the map
$\nu_2\colon F'_2\to F_2$ explicitly.  Since the matrix is too large
to be written down, we simply exhibit a basis element of $F'_2$ that
is not mapped to zero.

For $I=\{i_1<\dots <i_4\}$ a subset of either $\{1, \dots, 8\}$ or $\{1,
\dots, 7\}$, we use the notation $\epsilon_I:=\epsilon_{i_1}\wedge \dots
\wedge \epsilon_{i_4}$ to write $S$-bases for 
$S(-4)^{\binom{8}{4}}$ and for $S(-4)^{\binom{7}{4}}$.
Choose the natural monomial bases for the
cohomology groups appearing in the tensor product expressions for $F_2$ and $F_2'$, and write these monomials in multi-index notation.
Recalling the above definition of $h$, we then have that 
\[
\epsilon_{1,2,3,4}\otimes
\mathbf{y}^{\left(\begin{smallmatrix}-4&-1&0&1\\-1&-1&0&0
\end{smallmatrix}\right)}
\]
is a basis element of $F_2$.  We compute
\begin{align*}
\nu_2
\left(\epsilon_{1,2,3,4}\otimes
\mathbf{y}^{\left(\begin{smallmatrix}-4&-1&0&1\\-1&-1&0&0
\end{smallmatrix}\right)}\right)
& = \epsilon_{1,2,3,4}\otimes
\mathbf{y}^{\left(\begin{smallmatrix}-4&-1&0&1\\-1&-1&0&0
\end{smallmatrix}\right)}\cdot h\\
& = \epsilon_{1,2,3,4}\otimes
\mathbf{y}^{\left(\begin{smallmatrix}-4&-1&0&1\\-1&-1&0&0
\end{smallmatrix}\right)+\left(\begin{smallmatrix}1&0&2&2\\0&0&0&0
\end{smallmatrix}\right)}\\
& = \epsilon_{1,2,3,4}\otimes
\mathbf{y}^{\left(\begin{smallmatrix}-3&-1&2&3\\-1&-1&0&0
\end{smallmatrix}\right)}.
\end{align*}
Since this yields a basis element of $F_2'$, it is clear that $\nu_2$ is a nonzero map, 
so $\nu_\bullet$ is not null-homotopic.
\end{example}

\begin{proof}[Proof of Theorem~\ref{thm:deg:after:reduction}]
Construction~\ref{modif:es} yields 
finitely generated graded Cohen--Macaulay modules $M$ and $M'$ 
that have pure resolutions $F_\bullet$ and $F'_\bullet$ of types $d$ and
$d'$, respectively. 
To construct the desired nonzero map $\psi\colon M'\to M$, 
we fix a generic homogeneous form $h$ on $\PP$ of multidegree
$(0,a-a')$, which exists because $a-a' = \delta - \delta'\in\NN^r$. Multiplication by $h$ induces a map 
$h_\bullet\colon  \KK'_\bullet \to \KK_\bullet$. 
The functoriality of $\pi_*$ induces a map 
$\pi_*\KK'_\bullet\to\pi_*\KK_\bullet$ that, 
upon taking global sections in all twists, yields a map
$\nu_\bullet\colon  F_\bullet' \to F_\bullet$. 
Let $\psi \colon M'\to M$ be the map induced by $\nu_\bullet$. 

To show that $\psi$ is nonzero, it suffices to show that $\nu_\bullet$
is not null-homotopic.  Let $j$ be the index such that $d_j=d_j'$.
Then $F_j$ and $F_j'$ are generated entirely in the same degree.
Since $F_\bullet$ and $F_\bullet'$ are minimal free resolutions,
$\nu_j\colon F_j'\to F_j$ is given by a matrix of scalars.  Thus it
follows that $\nu_\bullet$ is null-homotopic only if $\nu_j$ is the
zero map.  We now use the description of $\nu_j$ given
in~\cite{explicit}*{Lemma 7.3}. (The relevant homological degree in
both $\KK_\bullet$ and $\KK'_\bullet$ is $d_j-d_0$.)

Let $s=\max\{i \mid a_i-d_j+d_0\leq -2\}$ and let $s'=\max\{i \mid
a_i'-d_j'+d_0\leq -2\}$.  Note that, since $d_j=d_j'$, the construction of
$a$ and $a'$ implies that $s=s'$.  We then have
\begin{align*}
  F_j & = S(-d_j)^{\binom{\ell(d)+r}{d_j-d_0}}\otimes \left(
    \bigotimes_{i=1}^s \Hs^1(\PP^1, \cO(a_i-d_j+d_0)) \right) \otimes
  \left( \bigotimes_{i=s+1}^r \Hs^0(\PP^1, \cO(a_i-d_j+d_0)) \right) \; \text{and} \\
  F'_j & = S(-d_j)^{\binom{\ell(d')+r}{d_j-d_0}}\otimes \left(
    \bigotimes_{i=1}^{s} \Hs^1(\PP^1, \cO(a'_i-d_j+d_0)) \right) \otimes
  \left( \bigotimes_{i=s+1}^{r} \Hs^0(\PP^1, \cO(a'_i-d_j+d_0)) \right),
\end{align*}
where both $F_j$ and $F'_j$ have the same number of factors involving
$\Hs^0$ (and therefore also the same number involving $\Hs^1$).  Hence
we can repeatedly apply~\cite{explicit}*{Lemma 7.3} to conclude that
$\nu_j$ is simply the map induced on cohomology by the map
$h_{d_j-d_0}\colon \mathcal K_{d_j-d_0}'\to \mathcal K_{d_j-d_0}$.

We now fix a specific value of $h$ and show that $\nu_{j}\ne 0$. 
Let $c:=a-a'\in \mathbb N^r$ and write $c=(c_1, \dots, c_r)$.  Let
\[
h:=\left(y_0^{(1)} \right)^{c_1} \cdot \left(y_0^{(2)} \right)^{c_2} \cdots \left(y_0^{(r)} \right)^{c_r}
=\mathbf{y}^{\left(\begin{smallmatrix}c_1& \dots& c_r\\ 0&\dots &0\end{smallmatrix}\right)},
\]
so that $h$ is the unique monomial of multidegree $(0,c)$ that
involves only the $y_0^{(i)}$-variables.

For $I=\{i_1<\dots <i_{d_j-d_0}\}$ a subset of either $\{1, \dots, \ell(d)+r\}$ or $\{1,
\dots, \ell(d')+r\}$, we use the notation $\epsilon_I:=\epsilon_{i_1}\wedge \dots
\wedge \epsilon_{d_j-d_0}$ to write $S$-bases for 
$S(-d_j)^{\binom{\ell(d)+r}{d_j-d_0}}$ and for $S(-d_j)^{\binom{\ell(d')+r}{d_j-d_0}}$.
Choose the natural monomial bases for the
cohomology groups appearing in the tensor product expression for $F_j$ and $F_j'$, and write these monomials in matrix
multi-index notation, as in Example~\ref{ex:degseqs}.
For each $i$ corresponding to an $\Hs^1$-term (i.e. $i\in \{1, \dots, s\}$), let $u_i:=-(a_i-d_j+d_0)+1$.  For each $i$ corresponding to an $\Hs^0$ term (i.e. $i\in \{s+1, \dots, r\}$), let $w_i:=-(a_i-d_j+d_0)$.  Observe that
\[
\epsilon_{\{1, \dots, d_j-d_0\}}\otimes \mathbf{y}^{\left(\begin{smallmatrix}u_1& \dots& u_s&w_{s+1}&\dots &w_r\\ -1&\dots&-1&0&\dots&0\end{smallmatrix}\right)}
\]
is a basis element of $F_j$.  We then have that
\begin{align*}
\nu_j
\left(
\epsilon_{\{1, \dots, d_j-d_0\}}\otimes 
\mathbf{y}^{\left(\begin{smallmatrix}u_1& \dots& u_s&w_{s+1}&\dots &w_r\\ -1&\dots&-1&0&\dots&0\end{smallmatrix}\right)}
\right)
& = \epsilon_{\{1, \dots, d_j-d_0\}}\otimes 
\mathbf{y}^{\left(\begin{smallmatrix}u_1& \dots& u_s&w_{s+1}&\dots &w_r\\ -1&\dots&-1&0&\dots&0\end{smallmatrix}\right)}
\cdot h\\
& = \epsilon_{\{1, \dots, d_j-d_0\}}\otimes 
\mathbf{y}^{\left(\begin{smallmatrix}u_1& \dots& u_s&w_{s+1}&\dots &w_r\\ -1&\dots&-1&0&\dots&0\end{smallmatrix}\right)}
\cdot \mathbf{y}^{\left(\begin{smallmatrix}c_1& \dots& c_r\\ 0&\dots &0\end{smallmatrix}\right)}\\
&=
\epsilon_{\{1, \dots, d_j-d_0\}}\otimes 
\mathbf{y}^{\left(\begin{smallmatrix}u_1+c_1& \dots& u_s+c_s&w_{s+1}+c_{s+1}&\dots &w_r+c_r\\ -1&\dots&-1&0&\dots&0\end{smallmatrix}\right)}.
\end{align*}
One may check that this is a basis element of $F_j'$, and hence the
map $\nu_j$ is nonzero. Therefore $\nu_\bullet$ is not null-homotopic,
as desired.
\end{proof}

\section{Equivariant construction of morphisms between 
modules with pure resolutions}
\label{sec:equiv:deg} 

Throughout this section, we assume that $\Bbbk$ is a field of
characteristic 0 and that all degree sequences have length $n$.  Let $V$ be
an $n$-dimensional $\Bbbk$-vector space, and let $S = \Sym(V)$.  We
use $\Sc_\lambda$ to denote a Schur functor, as in
Section~\ref{sec:equiv:root}. As in Section~\ref{sec:construct deg},
a shift of $d'$ reduces the remaining direction of
Theorem~\ref{thm:equivariant:deg} to the following result.
 
\begin{theorem} \label{theorem:eqvtmodules} Let $d \preceq d'$ be two
  degree sequences such that $d_k = d'_k$ for some $k$. Then there
  exist finite length  $\GL(V)$-equivariant $S$-modules $M$ and $M'$ 
  with pure resolutions of types $d$ and $d'$, respectively, with 
  $\Hom_{\GL(V)}(M',M)_0 \ne 0$.
\end{theorem}
Our proof of Theorem~\ref{theorem:eqvtmodules} relies on
Lemma~\ref{lemma:eqvtmodules}, which handles the special case when the
degree sequences $d$ and $d'$ differ by $1$ in a single position. This
proof will repeatedly appeal to Pieri's rule for decomposing the
tensor product of a Schur functor by a symmetric power. We refer the
reader to \cite[\S1.1 and Theorem 1.3]{sam} for a statement of this
rule, as our main use of it will be through \cite[Lemma 1.6]{sam}.

Given a degree sequence $d$, let $M(d)$ be the $\GL(V)$-equivariant
graded $S$-module constructed in~\cite{efw}*{\S3} (see also
\cite{sam}*{\S2.1}), and let ${\bf F}(d)_\bullet$ be its
$\GL(V)$-equivariant free resolution. By construction, the generators
for each $S$-module ${\bf F}(d)_j$ form an irreducible $\GL(V)$-module
whose highest weight we call $\lambda(d)_j$.  For instance, if
$d=(0,2,5,7,8)$, then $\lambda(d)_0=(3,1,0,0)$ and
$\lambda(d)_1=(5,1,0,0)$~\cite{efw}*{Example~3.3}.  Note that
$M(d)\otimes V$ is also an equivariant module with a pure resolution
of type $d$.
  
\begin{lemma} 
\label{lemma:eqvtmodules} 
Let $d = (d_0, \dots, d_n) \in \ZZ^{n+1}$ be a degree sequence, and
let $d'$ be the degree sequence obtained from $d$ by replacing $d_i$
by $d_i+1$ for some $i$. Then there exists an equivariant nonzero morphism 
$\phi \colon M(d')\otimes V\to M(d)$.

Further, if $F_\bullet$ and $F_\bullet'$ are the minimal free resolutions of $M(d)$ and $M(d')\otimes V$ respectively,
then we may choose $\phi$ so that the induced map $F_j'\to F_j$ is surjective for all $j\ne i$.
\end{lemma}

\begin{remark} \label{remark:eqvtefw} Let $d$ and $d'$ be degree sequences
as in the statement of Lemma~\ref{lemma:eqvtmodules}.  We observe that 
  \begin{enumerate}[(i)]
  \item $\lambda(d')_i = \lambda(d)_i$.
  \item If $j < i$, then $\lambda(d')_j$ is obtained from
    $\lambda(d)_j$ by removing a box from the $i$th part.
  \item If $j > i$, then $\lambda(d')_j$ is obtained from
    $\lambda(d)_j$ by removing a box from the $(i+1)$st part.
  \end{enumerate}
For instance, if $d=(0,2,4)$ and $d'=(0,3,4)$, then we have 
\[
\lambda(d)_j=\begin{cases}
(1,0) & \text{ if } j=0\\
(3,0) & \text{ if } j=1\\
(3,2) & \text{ if } j=2
\end{cases}
\quad \text{ and } \quad
\lambda(d')_j=\begin{cases}
(0,0) & \text{ if } j=0\\
(3,0) & \text{ if } j=1\\
(3,1) & \text{ if } j=2.
\end{cases} \qedhere
\]
\end{remark}

\begin{remark} In the proof of Lemma~\ref{lemma:eqvtmodules}, we
  repeatedly use~\cite{sam}*{Lemma 1.6}. The statement of the lemma is
  for factorizations of Pieri maps into simple Pieri maps $\Sc_\nu V
  \to \Sc_\eta V \otimes V$, but we need to factor into simple Pieri
  maps as well as simple co-Pieri maps $\Sc_\eta V \otimes V \to
  \Sc_\nu V$. No modification of the proof is needed: we simply use
  the fact that the composition of a co-Pieri map and a Pieri map of
  the same type is an isomorphism and that in each case that we
  apply~\cite{sam}*{Lemma 1.6}, the Pieri maps may be factored so that
  the simple Pieri maps and simple co-Pieri maps of the same type
  appear consecutively.
\end{remark}

\begin{proof}[Proof of Lemma~\ref{lemma:eqvtmodules}] 
  Set $\lambda_\ell = \sum_{j=\ell}^{n-1} (d_{j+1} - d_j - 1)$ for $1
  \le \ell \le n-1$, $\lambda_n = 0$, $\mu_1 = \lambda_1 + d_1 - d_0$,
  and $\mu_\ell = \lambda_\ell$ for $1 \le \ell \le n$. If $i = n$, we
  modify $\lambda$ and $\mu$ by adding 1 to all of its parts (so in
  particular, $\lambda_n = \mu_n = 1$).  As in~\cite{efw}*{\S3},
  define $M$ to be the cokernel of the Pieri map
  \[
  \psi_{\mu / \lambda} \colon  S(-d_1) \otimes \Sc_\mu V \to S(-d_0)
  \otimes \Sc_\lambda V. 
  \]
  We will choose partitions $\lambda'$ and $\mu'$ so that
  $M'$ is the cokernel of the Pieri map
  \[
  \psi_{\mu' / \lambda'} \colon  S(-d'_1) \otimes \Sc_{\mu'} V \to
  S(-d'_0) \otimes \Sc_{\lambda'} V.
  \]
  To do this, we separately consider the three cases $i=0$, $i=1$, and $i \ge
  2$. In each case, we specify $\lambda'$ and $\mu'$ (these
  descriptions are special cases of Remark~\ref{remark:eqvtefw}) and
  construct a commutative diagram of equivariant degree 0 maps
  \begin{equation} \label{eqn:square} \xymatrix{ S(-d_1) \otimes
      \Sc_\mu V \ar[rr]^-{\psi_{\mu/\lambda}}
      & & S(-d_0) \otimes \Sc_\lambda V \\
      S(-d'_1) \otimes \Sc_{\mu'} V \otimes V \ar[rr]^-{ \psi_{\mu' /
          \lambda'} \otimes 1_V} \ar[u]^-{\phi_\mu} & & S(-d'_0)
      \otimes \Sc_{\lambda'} V \otimes V \ar[u]^-{\phi_\lambda} }
  \end{equation}
  that induces an equivariant degree 0 map of the cokernels $\phi
  \colon  M' \to M$. Since the Pieri maps are only
  well-defined up to a choice of nonzero scalar, we only prove
  that the square commutes up to a choice of nonzero scalar. 
  One may scale appropriately to obtain strict commutativity. 
%
  
Finally, after handling the three separate cases, we prove that the induced maps 
$F_j'\to F_j$ are surjective whenever $j\ne i$.  Since $F_\bullet'$ is a minimal
free resolution, this implies that the map $F_\bullet'\to F_\bullet$ is not null-homotopic,
and hence $\phi \colon M'\to M$ is nonzero.


  ~

  \noindent \emph{Case $i=1$.} Set $\lambda'_1 = \lambda_1 - 1$,
  $\lambda'_j = \lambda_j$ for $2 \le j \le n$, and $\mu' =
  \mu$. Also, let $d'_0 = d_0$ and $d'_1 = d_1 + 1$. Using the
  notation of \eqref{eqn:square}, we define $\phi_\mu$ by identifying
  $\Sc_{\mu'} V \otimes V$ with $\Sym^1 V \otimes \Sc_\mu V$ and then
  extending it to an $S$-linear map. Let $\phi_\lambda$ be 
  the projection of $\Sc_{\lambda'} V \otimes V \to \Sc_\lambda V$
  tensored with the identity of $S(-d_0)$. From the degree $d_1 + 1$
  part of \eqref{eqn:square}, we obtain
  \[
  \xymatrix{ \Sym^1 V \otimes \Sc_\mu V \ar[r]^-\alpha & \Sym^{d_1 -
      d_0 + 1} V  \otimes \Sc_\lambda V \\
    \Sc_\mu V \otimes V \ar[r]^-\delta \ar[u]^-\beta & \Sym^{d_1 - d_0
      + 1} V \otimes \Sc_{\lambda'} V \otimes V \ar[u]^-\gamma. }
  \]
  Note that $\alpha$ is the linear part of ${\bf F}_1 \to {\bf F}_0$
  and is hence injective because $d_2 - d_1 > 1$. 
  Since $\beta$ is an isomorphism, $\alpha \beta$ is injective. 
  Also we have $\lambda_1 >
  \lambda_2$ because $d_2 - d_1 > 1$, so by Pieri's rule, every summand
  of $\Sc_\mu V \otimes V$ is also a summand of $\Sym^{d_1 - d_0 + 1}
  V \otimes \Sc_\lambda V$. Using~\cite{sam}*{Lemma~1.6}, one can show
  that $\gamma \delta$ is also injective. Since the tensor product
  $\Sym^{d_1 - d_0 + 1} V \otimes \Sc_\lambda V$ is multiplicity-free
  by the Pieri rule, this implies that these maps are equal after
  rescaling the image of each direct summand of $\Sc_\mu V \otimes V$
  by some nonzero scalar. Hence this diagram is commutative, and the
  same is true for \eqref{eqn:square}.

  ~

  \noindent \emph{Case $i \ge 2$.} Set $\lambda'_i =
  \lambda_i - 1$ and $\lambda_j = \lambda_j$ for $j \ne i$. 
  Similarly, set $\mu'_i = \mu_i - 1$ and $\mu'_j = \mu_j$ for $j \ne i$. 
  Using the notation of \eqref{eqn:square}, let $\phi_\mu$ 
  be a nonzero projection of $\Sc_{\mu'} V \otimes V$ onto $\Sc_\mu V$
  tensored with the identity on $S(-d_1)$. Similar to the previous case, 
  choose a nonzero projection $\Sc_{\lambda'} V \otimes V \to \Sc_\lambda V$
  and tensor it with the identity map on $S(-d_0)$ to get
  $\phi_\lambda$. From the degree $d_1$ part of
  \eqref{eqn:square}, we obtain
  \[
  \xymatrix{ \Sc_\mu V \ar[r]^-\alpha & \Sym^{d_1 - d_0} V \otimes
    \Sc_\lambda V \\
    \Sc_{\mu'} V \otimes V \ar[r]^-\delta \ar[u]^-\beta & \Sym^{d_1 -
      d_0} V \otimes \Sc_{\lambda'} V \otimes V \ar[u]^-\gamma. }
  \]
  Let $\Sc_\nu V$ be a direct summand of $\Sc_{\mu'} V \otimes V$. If
  $\nu \ne \mu$, then $\Sc_\nu V$ is not a summand of $\Sym^{d_1 -
    d_0} V \otimes \Sc_\lambda V$, as otherwise we would have $\nu_i =
  \lambda_i - 1$, and both of the compositions $\alpha \beta$ and $\gamma
  \delta$ would therefore be 0 on such a summand. If $\nu = \mu$, then the
  composition $\alpha \beta$ is nonzero, so it is enough to check that
  the same is true for $\gamma \delta$; this holds
  by~\cite{sam}*{Lemma~1.6}, and hence this diagram and
  \eqref{eqn:square} are commutative.
  
  ~
  
  \noindent \emph{Case $i=0$.} Set $d^\vee:= (-d_n, -d_{n-1}, \dots,
  -d_0)$ and $d'^{\vee}:=(-d'_n, -d'_{n-1}, \dots, -d'_0)$.  Since
  $d_j=d_j'$ for all $j\ne i=0$, we see that $d^\vee$ and $d'^\vee$
  only differ in position $n$.  Hence, by the case $i\geq 2$ above (we
  assume that $n \ge 2$ since the $n=1$ case is easily done directly),
  we have finite length modules $M(d^\vee)$ and $M(d'^{\vee})$ with
  pure resolutions of types $d^\vee$ and $d'^\vee$, respectively,
  along with a nonzero morphism $\psi \colon M(d^\vee)\otimes V \to
  M(d'^\vee)$.  If we define $N^\vee:= \Ext^n(N,S),$ then
  $M(d'^\vee)^\vee \cong M(d')$ and $(M(d^\vee)\otimes V)^\vee \cong
  M(d) \otimes V^*$ (both isomorphisms are up to some power of
  $\bigwedge^n V$ which we cancel off). In addition, since
  $\Ext^n(-,S)$ is a duality functor on the space of finite length
  $S$-modules, we obtain a nonzero map
  \[
  \psi^\vee \colon M(d')\to M(d)\otimes V^*.
  \]
  By adjunction, we then obtain a nonzero map $M(d')\otimes V\to
  M(d)$.

  ~
    
  Fixing some $j\ne i$, we now prove the surjectivity of the maps $F_j'\to F_j$, which implies that $\phi$ is a nonzero morphism, as observed above.  The key observation is that,
  in each of the above three cases, $F_j$ is an irreducible Schur module.  Since $d_j=d_j'$, the map
  \[
  F_j'=S(-d'_j)\otimes \Sc_{\lambda(d')_j}V\otimes V\to F_j=S(-d_j)\otimes \Sc_{\lambda(d)_j}V
  \]
  is induced by a nonzero equivariant map $\Sc_{\lambda(d')_j}V\otimes V\to  \Sc_{\lambda(d)_j}V$.  Since the target
  is an irreducible representation, this morphism, and hence the map $F_j'\to F_j$, is surjective.  More specifically, the map  $\Sc_{\lambda(d')_j}V\otimes V\to  \Sc_{\lambda(d)_j}V$ is a projection onto one of the factors in the Pieri rule decomposition of $\Sc_{\lambda(d')_j}V\otimes V$.
\end{proof}

\begin{example} \label{example:eqvtres1}
  This example illustrates the construction of
  Lemma~\ref{lemma:eqvtmodules} when $d=(0,2,4)$ and $d'=(0,3,4)$.
  When writing the free resolutions, we simply write the Young
  diagram of $\lambda$ in place of the corresponding graded 
  equivariant free module. Also, we follow the conventions in
  \cite{efw} and \cite{sam} and draw the Young diagram of $\lambda$ by
  placing $\lambda_i$ boxes in the $i$th {\it column}, rather than the
  usual convention of using rows. The morphism from
  Lemma~\ref{lemma:eqvtmodules} yields a map of complexes, which we write as  
\[
\begin{CD}
  @.M@<<< {\tiny \tableau[scY]{|}} @<<< {\tiny \tableau[scY]{|||}}
  @<<< {\tiny
    \tableau[scY]{,|,||}} @<<< 0 \\
  @. _{\psi} @AAA@AAA @AAA @AAA  \\
  @. M'@<<< {\tiny \tableau[scY]{|}}\otimes \varnothing @<<< {\tiny
    \tableau[scY]{|}} \otimes {\tiny \tableau[scY]{|||}} @<<< {\tiny
    \tableau[scY]{|}} \otimes {\tiny \tableau[scY]{,|||}} @<<< 0.\\
\end{CD}
\]
Observe that $d_2=4=d_2'$ and that the vertical arrow in homological position $2$ is surjective, as it corresponds to a Pieri rule projection.  A similar statement holds in position $0$.
\end{example}

\begin{proof}[Proof of Theorem~\ref{theorem:eqvtmodules}] 
Set $r:=\sum_{j=0}^n d_j'-d_j$.  We may construct a sequence of degree sequences
$d=:d^0<d^1<\dots <d^r:=d'$ such that $d^j$ and $d^{j+1}$ satisfy the hypotheses of
Lemma~\ref{lemma:eqvtmodules} for any $j$.  
Lemma~\ref{lemma:eqvtmodules} yields a nonzero morphism
\[
\phi^{(j+1)} \colon M(d^{j+1})\otimes V\to M(d^j)
\]
for any $j=1, \dots, r$.
If we set $M^{(j)}:=M(d^j)\otimes V^{\otimes j}$, and we set $\psi^{(j+1)}$ to be the natural map
\[
\psi^{(j+1)} \colon M^{(j+1)}\to M^{(j)}
\]
given by $\phi^{(j)}\otimes \text{id}_V^{\otimes j}$, then we may compose the map $\psi^{(j+1)}$ with
the map $\psi^{(j)}$.

Let $M:=M^{(0)}=M(d),$ and let $M':=M^{(r)}=M(d')\otimes V^{\otimes r}$.
We then have an equivariant map $\psi:=\psi^{(1)}\circ \dots \circ
\psi^{(r)} \colon M'\to M$, and we must finally show that $\psi$ is
nonzero.  Let $F^{(j)}_\bullet$ be the minimal free resolution of
$M^{(j)}$.  Since $d_k=d_k'$, it follows that $d^{(j)}_k=d^{(j+1)}_k$
for all $j$.   Lemma~\ref{lemma:eqvtmodules} then implies that we can
choose each $\phi^{(j+1)}$ such that the map $\psi^{(j+1)}$ induces a
surjection $F^{(j+1)}_k\to F^{(j)}_k$.  Since the composition of
surjective maps is surjective, it follows that the map $F^{(r)}_k\to
F^{(0)}_k$ induced by $\psi$ is surjective.  Since $F^{(0)}_\bullet$
is a minimal free resolution, we conclude that the map of
complexes $F^{(r)}_\bullet\to F^{(0)}_\bullet$ is not 
null-homotopic, and hence $\psi \colon M'\to M$ is a nonzero morphism.
\end{proof}

\begin{remark}\label{rmk:simpler}
  By introducing a variant of Lemma~\ref{lemma:eqvtmodules}, we may
  simplify the construction used in the proof of
  Theorem~\ref{theorem:eqvtmodules}.  Let $d$ and $d'$ be two degree
  sequences such that $d_i'=d_i+N$, and $d_j'=d_j$ for all $j\ne i$.
  Iteratively applying Lemma~\ref{lemma:eqvtmodules} yields a morphism
  $\phi \colon M(d')\otimes V^{\otimes N}\to M(d)$.  Since
  $\ch(\Bbbk)=0$, we have an inclusion $\iota \colon \Sym^N V\to
  V^{\otimes N}$, and we let $\psi$ be the morphism induced by
  composing $\phi$ and $\text{id}_{M(d')}\otimes \iota$.  Let
  $F_\bullet'$ and $F_\bullet$ be the minimal free resolutions of
  $M(d')\otimes \Sym^N V$ and $M(d)$ respectively.  The map $F_j'\to
  F_j$ induced by $\psi$ is induced by the equivariant map of vector spaces
\[
\Sc_{\lambda(d')}V\otimes \Sym^NV\to \Sc_{\lambda(d)}V.
\]
This map is surjective because it is a projection onto one of the factors in the Pieri rule decomposition of
$\Sc_{\lambda(d')}V\otimes \Sym^NV$.  

This simplifies the proof of Theorem~\ref{theorem:eqvtmodules} as follows.  Let $i_1 > \cdots > i_\ell$ be the indices for which $d$ and $d'$ differ. By iteratively applying the construction outlined in this remark, we may  construct the desired modules and nonzero morphism in $\ell$ steps.  Since $\ell$ can be far smaller than $r:=\sum_{j=0}^n d_j'-d_j$, this variant is useful for computing examples such as Example~\ref{ex:big eqvt}.
\end{remark}

\begin{example}\label{ex:big eqvt} 
We illustrate Theorem~\ref{theorem:eqvtmodules}
with $n=4$, $d = (0,2,3,6,7)$, and $d' = (1,2,5,6,10)$.  Using the notation
  of Remark~\ref{rmk:simpler}, $d^{(1)} = (0,2,3,6,10)$, $d^{(2)} = (0,2,5,6,10)$. 
  Following the same conventions as in Example~\ref{example:eqvtres1}, 
  the corresponding resolutions are given in Figure~\ref{fig:big diagram}.
    Notice that $d_3=6=d_3'$.  Focusing on the third terms of the 
  resolutions, we see that the maps are simply projections from 
  Pieri's rule. In particular, these maps are surjective and therefore 
  nonzero.
\qedhere

\begin{landscape}
\begin{figure}
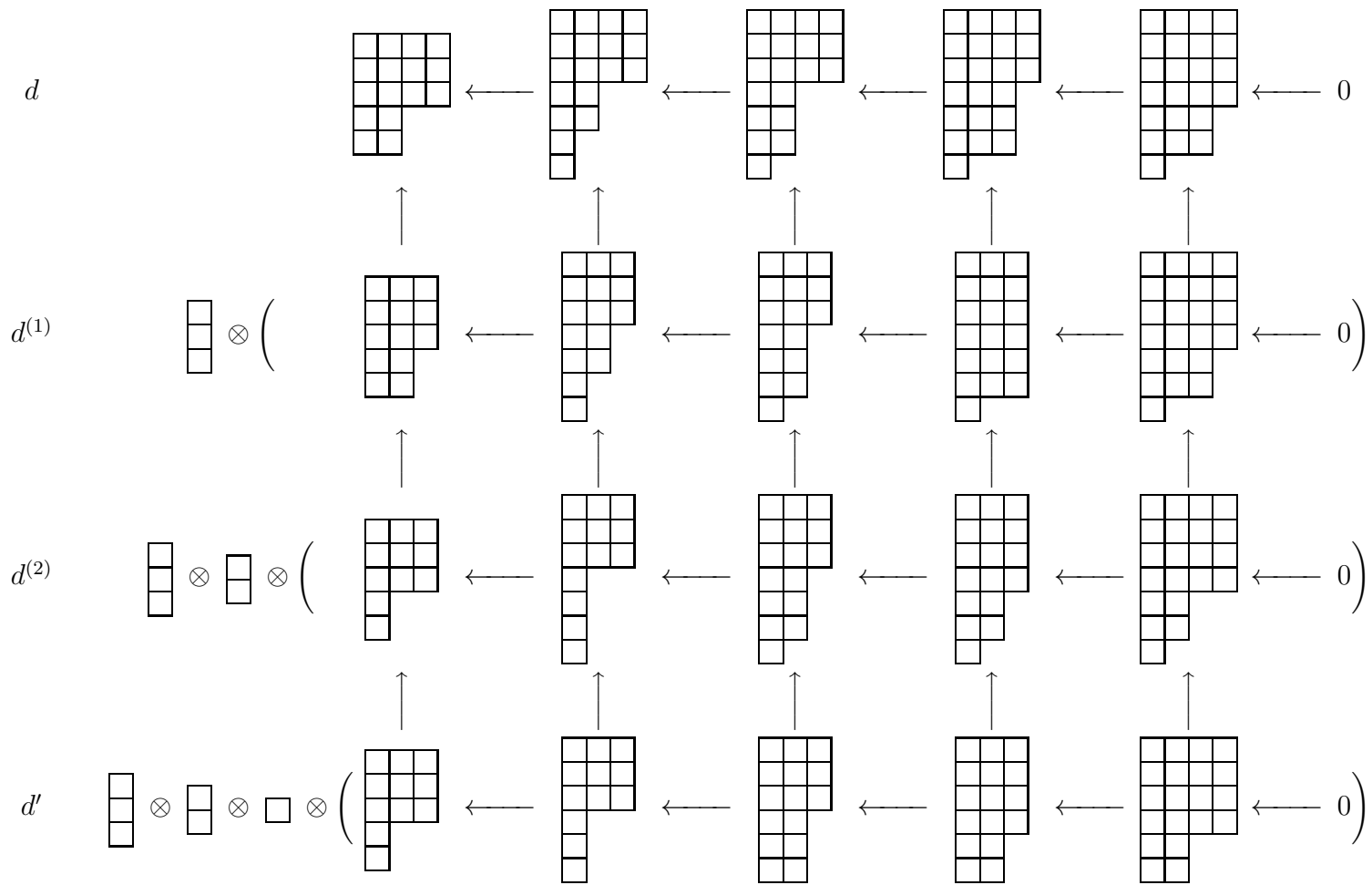

\vspace*{1cm}
  \[
  \begin{CD}
    d \quad \quad @. @.  {\tiny \tableau[scY]{,,,|,,,|,,,|,|,}} @<<<
    {\tiny \tableau[scY]{,,,|,,,|,,,|,|,|||}} @<<< {\tiny
      \tableau[scY]{,,,|,,,|,,,|,|,|,||}} @<<< {\tiny
      \tableau[scY]{,,,|,,,|,,,|,,|,,|,,||}} @<<< {\tiny
      \tableau[scY]{,,,|,,,|,,,|,,,|,,|,,||}} @<<< 0 \\
    @. @. @AAA @AAA @AAA @AAA @AAA \\
    d^{(1)} \quad \quad @. {\tiny \tableau[scY]{|||}} \ \otimes \bigg(
    @. {\tiny \tableau[scY]{,,|,,|,,|,|,}} @<<<{\tiny
      \tableau[scY]{,,|,,|,,|,|,|||}} @<<< {\tiny
      \tableau[scY]{,,|,,|,,|,|,|,||}} @<<< {\tiny
      \tableau[scY]{,,|,,|,,|,,|,,|,,||}} @<<< {\tiny
      \tableau[scY]{,,,|,,,|,,,|,,,|,,|,,||}} @<<< 0 @. \bigg) \\
    @. @. @AAA @AAA @AAA @AAA @AAA \\
    d^{(2)} \quad \quad @. {\tiny \tableau[scY]{|||}} \ \otimes \
    {\tiny \tableau[scY]{||}} \ \otimes \bigg( @. {\tiny
      \tableau[scY]{,,|,,|,,|||}} @<<<{\tiny
      \tableau[scY]{,,|,,|,,|||||}} @<<< {\tiny
      \tableau[scY]{,,|,,|,,|,|,|,||}} @<<< {\tiny
      \tableau[scY]{,,|,,|,,|,,|,|,||}} @<<< {\tiny
      \tableau[scY]{,,,|,,,|,,,|,,,|,|,||}} @<<< 0 @. \bigg) \\
    @. @. @AAA @AAA @AAA @AAA @AAA \\
    d' \quad \quad @. {\tiny \tableau[scY]{|||}} \ \otimes \ {\tiny
      \tableau[scY]{||}} \ \otimes \ {\tiny \tableau[scY]{|}} \
    \otimes \bigg( @. {\tiny \tableau[scY]{,,|,,|,,|||}} @<<<{\tiny
      \tableau[scY]{,,|,,|,,||||}} @<<< {\tiny
      \tableau[scY]{,,|,,|,,|,|,|,|}} @<<< {\tiny
      \tableau[scY]{,,|,,|,,|,,|,|,|}} @<<< {\tiny
      \tableau[scY]{,,,|,,,|,,,|,,,|,|,|}} @<<< 0 @. \bigg) \\
  \end{CD}
  \]
  \smallskip
  \caption{The Young diagram depictions of the resolutions in
    Example~\ref{ex:big eqvt}.}
  \label{fig:big diagram}
  \end{figure}
\end{landscape}
\end{example}

\section{The poset of root sequences}
\label{sec:prelim root}

Let $\EE$ be a coherent sheaf on $\PP^{n-1}$. 
The \defi{cohomology table} of $\EE$ is a table 
with rows indexed by $\{0, \ldots, n-1\}$ and columns indexed by $\ZZ$, 
such that the entry in row $i$ and column $j$ is 
$\dim_\Bbbk \Hs^i(\PP^{n-1}, \EE(j-i))$. 
A sequence 
$f=(f_1, \dots, f_{n-1})\in \left(\ZZ\cup\{-\infty\}\right)^{n-1}$ 
is called a \defi{root sequence} for $\PP^{n-1}$
if $f_i<f_{i-1}$ for all $i$ (with the convention that $-\infty<-\infty$).  
The \defi{length} of $f$, denoted $\ell(f)$, is the largest integer $t$
such that $f_t$ is finite.
\begin{definition}\label{def:supernatural}
Let $f$ be a root sequence for $\PP^{n-1}$. A sheaf $\EE$ on $\PP^{n-1}$ is
\defi{supernatural of type} $f=(f_1, \dots, f_{n-1})$ if the following are satisfied: 
\begin{asparaenum}
\item The dimension of $\Supp \EE$ is $\ell(f)$.
\item For all $j\in \mathbb Z$, there exists at most one $i$ 
		such that $\dim_\Bbbk \Hs^i(\PP^{n-1}, \EE(j))\ne 0$.
\item The Hilbert polynomial of $\EE$ has roots $f_1, \dots, f_{\ell(f)}$.
\end{asparaenum}
Dropping the reference to its root sequence, 
we also say that $\EE$ is a \defi{supernatural sheaf} 
(or a \defi{supernatural vector bundle} if it is locally free). 
\end{definition}

For every root sequence $f$, there exists a supernatural sheaf of type
$f$~\cite{EiScConjOfBS07}*{Theorem~0.4}. 
Moreover, the cohomology table of any coherent sheaf 
can be written as a positive real combination of cohomology tables 
of supernatural sheaves~\cite{EiScSupNat09}*{Theorem~0.1}.
The \defi{cone of cohomology tables} for $\PP^{n-1}$ 
is the convex cone inside $\prod_{j \in \ZZ}\RR^n$ 
generated by cohomology tables of coherent sheaves on $\PP^{n-1}$.  
Each root sequence $f$ corresponds to a unique extremal ray of this cone, 
which we denote by $\rho_f$, and  
every extremal ray is of the form $\rho_f$ for some root sequence $f$.

\begin{definition}\label{defn:partial:root}
For two root sequences $f$ and $f'$, we say that $f\preceq f'$ and that
$\rho_f \preceq \rho_{f'}$ if $f_i\leq f_i'$ for all $i$.
\end{definition}

This partial order induces a simplicial fan structure on the cone of cohomology
tables, where simplices correspond to chains of root sequences under the partial
order $\preceq$.
We now
show that the existence of a nonzero homomorphism between two
supernatural sheaves implies the comparability of their corresponding
root sequences, which provides the reverse implications for
Theorems~\ref{thm:poset:root:main} and~\ref{thm:equivariant:root}.

\begin{prop}\label{prop:half:poset:root}
Let $\EE$ and $\EE'$ be supernatural sheaves of types $f$ and $f'$ 
respectively.  If $\Hom(\EE',\EE)\ne 0$, then $f\preceq f'$. 
\end{prop}
\begin{proof}
  Let $\tate(\EE)$ and $\tate(\EE')$ denote the Tate resolutions of
  $\EE$ and $\EE'$~\cite{EiFlScExterior03}*{\S4}.  These are doubly
  infinite acyclic complexes over the exterior algebra $\Lambda$,
  which is Koszul dual to $S$ and has generators in degree $-1$.
  Since $\Hom(\EE',\EE)\ne 0$, there is a map $\phi\colon \tate(\EE')
  \to \tate(\EE)$ that is not null-homotopic.  Observe that for every
  cohomological degree $j$, $\phi^j \colon \tate(\EE')^j \to
  \tate(\EE)^j$ is nonzero.  First, if $\phi^j=0$ for some $j$, then,
  we may take $\phi^k=0$ for all $k < j$. Secondly, if $k > j$, then
  after applying $\Hom_{\Lambda}(-, \Lambda)$ (which is exact because
  $\Lambda$ is self-injective), we can take $\phi^k$ to be zero.

By~\cite{EiScConjOfBS07}*{Theorem~6.4}, we see that all the minimal
generators of $T(\EE)^j$ (respectively, $T(\EE')^j$) are of a single degree
$i$ (respectively, $i'$). (This is equivalent to stating that every
column of the cohomology table of $\EE$ and $\EE'$ contains precisely one
nonzero entry.) Since $\phi^j$ is nonzero and $\Lambda$ is generated in
elements of degree $-1$, we see that $i' \leq i$. Now,  again
by~\cite{EiScConjOfBS07}*{Theorem~6.4},  $f \preceq f'$.
\end{proof}

\section{Construction of morphisms between supernatural sheaves}
\label{sec:construct root}

The goal of this section is to prove Theorem~\ref{thm:main:root}, which provides
the forward direction of Theorem~\ref{thm:poset:root:main}.

\begin{thm}
\label{thm:main:root}
Let $f\preceq f'$ be two root sequences. 
Then there exist supernatural sheaves 
$\EE$ and $\EE'$ of types $f$ and $f'$, respectively, 
with $\Hom(\EE',\EE)\ne 0$.
\end{thm}

For the purposes of exposition, we separate the proof of
Theorem~\ref{thm:main:root} into two cases (with $\ell(f) = \ell(f')$
and with $\ell(f) < \ell(f')$), and handle these cases in
Propositions~\ref{propn:rootSameLength} and~\ref{propn:rootDiffLength}
respectively. Examples~\ref{ex:1} and ~\ref{ex:root:diff:lengths}
illustrate the essential ideas behind the proof in each case.

If $\ell(f) < n-1$, then we call $(f_1, \ldots, f_{\ell(f)})$ 
the \defi{truncation} of $f$, and write $\tau(f)$.  Let $f=(f_1, \dots,
f_{n-1})$ be a root sequence with $\ell(f)=s$.  Denote the $s$-fold product
of $\PP^1$ by $\PP^{1\times s}$.  Fix homogeneous
coordinates 
\begin{equation}
\label{equation:homogCoordProdPP}
\left([y_0^{(1)}:y_1^{(1)}], \dots, [y_0^{(s)}:y_1^{(s)}]\right)
\quad\text{on}\quad \PP^{1\times s}.
\end{equation}

In order to produce a supernatural sheaf of type $f$ on $\PP^{n-1}$, 
we first construct a supernatural vector bundle of type $\tau(f)$ on $\PP^s$. 
Its image under an embedding of $\PP^s$ as a linear subvariety $\PP^{n-1}$
will give the desired supernatural sheaf.

We now outline our approach to construct a nonzero map between supernatural sheaves on $\PP^s$ of types $f\preceq f'$ in the case that $\ell(f) = \ell(f') = s$. 
This uses the proof of~\cite{EiScConjOfBS07}*{Theorem~6.1}. 
\begin{enumerate}
\item\label{enum:snvbConstEqLenFinMap} 
	Construct a finite map $\pi\colon  \PP^{1\times s}\to \PP^{s}$.
\item\label{enum:snvbConstEqLenPushFwd} 
	Choose appropriate line bundles $\mathcal L$ and $\mathcal L'$ on $\PP^{1\times s}$ 
	so that $\pi_* \mathcal L$ and $\pi_*\mathcal L'$ are supernatural vector bundles
	of the desired types.
\item\label{enum:snvbConstEqLenMaps} 
	When $\ell(f) = \ell(f') = s$, construct a morphism $\mathcal L' \stackrel{\phi}{\longrightarrow}
    \mathcal L$ such that $\pi_* \phi$ is nonzero.
\end{enumerate}
For \eqref{enum:snvbConstEqLenFinMap}, 
we use the multilinear $(1,\dots,1)$-forms 
\begin{equation}\label{eqn:fp}
g_p:= \sum_{i_1+\dots+i_{s}=p} \left( \prod_{j=1}^{s} y_{i_j}^{(j)}
\right) \qquad \text{for}\quad p=0, \dots, s
\end{equation}
on $\PP^{1\times s}$ to define the map 
$\pi\colon \PP^{1\times s}\to \PP^s$ via $[g_0:\cdots:g_s]$.
For \eqref{enum:snvbConstEqLenPushFwd}, 
with $\mathbf 1 :=(1,\dots, 1)\in\ZZ^s$,
\[
\EE_f:=\pi_*\left(\cO_{\PP^{1\times s}}(-f-\mathbf 1)\right) 
\]
is a supernatural vector bundle of type $\tau(f)$ on $\PP^s$ of rank
$s!$ (the degree of $\pi$). The next example illustrates
\eqref{enum:snvbConstEqLenMaps}.

\begin{example}
\label {ex:1}
Here we find a nonzero morphism $\EE_{f'}\to \EE_{f}$ that 
is the direct image of a morphism of line bundles on $\PP^{1\times (n-1)}$. 
Let $n=5$ and $f:=(-2,-3,-4,-5) \preceq f':=(-1,-2,-3,-4)$. 
The map $\pi\colon \PP^{1\times 4}\to \PP^4$ is finite of degree $4!=24$. 
Following steps \eqref{enum:snvbConstEqLenFinMap} and \eqref{enum:snvbConstEqLenPushFwd} as outlined above, we set 
$\EE:=\EE_{f}=\pi_* \cO_{\PP^{1\times 4}}(1,2,3,4)$ and 
$\EE':=\EE_{f'}=\pi_* \cO_{\PP^{1\times 4}}(0,1,2,3)$. 
There is a natural inclusion
\begin{equation}
\label{equation:inclGlobalSec}
\pi_* \sheafHom_{\PP^{1\times 4}}\left( \cO_{\PP^{1\times 4}}(0,1,2,3),
\cO_{\PP^{1\times 4}}(1,2,3,4)\right)\subseteq
\sheafHom_{\PP^4}\left( \EE', \EE \right),
\end{equation}
which induces an inclusion of global sections 
(see Remark~\ref{rmk:inclusion}). Therefore
\begin{align*}
  \Hom(\EE', \EE)&\supseteq \Hs^0\left(\PP^4, \pi_* \sheafHom_{\PP^{1\times 4}}\left( \cO_{\PP^{1\times 4}}(0,1,2,3), \cO_{\PP^{1\times 4}}(1,2,3,4)\right) \right)\\
  &= \Hs^0(\PP^{1\times 4}, \cO_{\PP^{1\times 4}}(1,1,1,1))\\
  &\simeq\Bbbk^{16}.
\end{align*}
We thus conclude that $\Hom(\EE', \EE)\ne 0$.

The inclusion \eqref{equation:inclGlobalSec} is strict. 
Note that, by definition, neither $\EE'$ nor $\EE$ has intermediate cohomology, 
and hence, by Horrocks' Splitting Criterion, both $\EE$ and $\EE'$ must split as the sum of line bundles. 
Thus $\EE'=\cO_{\PP^4}^{24}$ and $\EE=\cO_{\PP^4}(1)^{24}$, and it follows 
that $\Hom(\EE',\EE)=\Hs^0(\PP^4,\cO(1)^{576}) \simeq \Bbbk^{2880}$.
\end{example}

\begin{remark}\label{rmk:inclusion}
Let $\pi\colon\PP^{1\times s} \to \PP^s$ be as in \eqref{enum:snvbConstEqLenFinMap}. 
For coherent sheaves $\FF$ and $\GG$ on $\PP^{1\times s}$, we have 
\[
\pi_* \sheafHom_{\cO_{\PP^{1\times s}}}(\FF,\GG)
\subseteq
\sheafHom_{\cO_{\PP^s}}(\pi_*\FF,\pi_*\GG).
\]
Indeed, this can be checked locally. 
Let $U \subseteq \PP^s$ be an affine open subset, 
and write $A = \Hs^0(U, \cO_{\PP^s})$ and 
$B = \Hs^0(U, \pi_*\cO_{\PP^{1\times s}})$. 
For all $B$-modules $M$ and $N$, every nonzero $B$-module 
homomorphism is also a nonzero $A$-module homomorphism 
via the map $A \rightarrow B$.  Injectivity is immediate.
\end{remark}

\begin{remark}
\label{remark:pushFwdSupNat}
Suppose that $\beta \colon \PP^s \to \PP^{n-1}$ is a closed immersion 
as a linear subvariety. 
Let $\EE$ be a coherent sheaf on $\PP^s$.
It follows from the projection formula and from the finiteness of $\beta$ that
$\EE$ is a supernatural sheaf
on $\PP^s$ of type $(f_1, \ldots, f_s)$ if and only if $\beta_*\EE$ is a
supernatural sheaf on $\PP^{n-1}$ of type 
$(f_1, \ldots, f_s, -\infty, \ldots, -\infty)$.
\end{remark}

\begin{prop}
\label{propn:rootSameLength}
If $\ell(f) = \ell(f')$, then Theorem~\ref{thm:main:root} holds.
\end{prop}

\begin{proof}
We first reduce to the case $\ell(f')=n-1$. 
Let $\beta\colon\PP^{\ell(f')} \to \PP^{n-1}$ be a closed immersion as a
linear subvariety. 
Let $\ell(f')=s$ and write $f=(f_1, \dots, f_s, -\infty, \dots, -\infty)$ and
$f'=(f_1', \dots, f_s', -\infty, \dots, -\infty)$. 
Assume that $\EE$ and $\EE'$ are supernatural sheaves 
of type $(f_1, \dots, f_s)$ and $(f'_1, \dots, f_s')$ on $\PP^{s}$ and that 
$\Hom(\EE', \EE)\ne 0$.  Then, by Remark~\ref{remark:pushFwdSupNat}, 
$\beta_*\EE$ and $\beta_*\EE'$ are supernatural sheaves of types $f$ and $f'$, 
and $\Hom(\beta_*\EE', \beta_*\EE) \neq 0$.

We may thus assume that $\ell(f')=n-1$.  
Let $\mathbf{1} := (1,\dots,1) \in\ZZ^{n-1}$.  
Let $\pi\colon \PP^{1 \times (n-1)} \to \PP^{n-1}$ be the morphism given by 
the forms $g_p$ defined in~\eqref{eqn:fp} (with $s=n-1$). 
Let $\EE := \EE_f=  \pi_* \cO(-f-\mathbf 1)$ and 
$\EE' := \EE_{f'}= \pi_*\cO(-f'-\mathbf{1})$. 
Remark~\ref{rmk:inclusion} shows that  
\[
\Hs^0\left(\PP^{n-1}, \pi_* \sheafHom_{\PP^{1\times (n-1)}}\left(
\cO(-f'-\mathbf{1}), \cO(-f-\mathbf{1}) \right) \right) \subseteq
\Hom_{\PP^{n-1}}(\EE',\EE).
\]
Note that
$\sheafHom_{\PP^{1\times (n-1)}}\left(
\cO(-f'-\mathbf{1}), \cO(-f-\mathbf{1}) 
\right) = \cO( f' - f )$. 
Since $f \preceq f'$, we have that 
$\Hs^0(\PP^{1\times (n-1)}, \cO( f'- f ))\ne 0$, 
and thus $\Hom_{\PP^{n-1}}(\EE',\EE) \neq 0$.
\end{proof}

When $\ell(f) < \ell(f')$, the supernatural sheaves 
constructed using~\eqref{enum:snvbConstEqLenFinMap} 
and~\eqref{enum:snvbConstEqLenPushFwd} above have supports 
of different dimensions. 
Before addressing this general case, we provide an example. 

\begin{example}
\label{ex:root:diff:lengths}
Let $n=5$ and $f=(-2,-3,-4,-\infty) \preceq f'=(-1,-2,-3,-4)$, 
so that $\ell(f)=3 < \ell(f') = 4 = n-1$. 
We proceed by modifying steps 
\eqref{enum:snvbConstEqLenFinMap}-\eqref{enum:snvbConstEqLenMaps} 
above. 
{
\makeatletter
\def\theenumi{\@roman\c@enumi$'$}
\makeatother
\begin{enumerate}
\item\label{enum:snvbConstNonEqLenFinMap}
We extend the construction of \eqref{enum:snvbConstEqLenFinMap} 
to the commutative diagram
\[
\xymatrix{
\PP^{1\times 3} \ar[r]^{\alpha} \ar[d]^{\pi^{(3)}} & \PP^{1 \times 4} \ar[d]^{\pi^{(4)}}\\
\PP^3 \ar[r]^{\beta}&\PP^{4}.
}
\]
\item\label{enum:snvbConstNonEqLenPushFwd} 
Choose appropriate line bundles $\mathcal L$ on $\PP^{1\times 3}$ 
and $\mathcal L'$ on $\PP^{1\times 4}$, 
so that $\pi^{(3)}_* \mathcal L$ and $\pi^{(4)}_* \mathcal L'$ 
are supernatural sheaves of the desired types.
\item\label{enum:snvbConstNonEqLenMaps}
	Construct a morphism 
	$\mathcal L' \stackrel{\phi}{\longrightarrow}\alpha_*\mathcal L$ 
	such that $\pi^{(4)}_* \phi$ is nonzero. 
\end{enumerate}
}

For \eqref{enum:snvbConstNonEqLenFinMap}, we use the 
homogeneous coordinates from \eqref{equation:homogCoordProdPP}.
The maps $\pi^{(3)}$ and $\pi^{(4)}$ are instances of the map $\pi$ 
from \eqref{enum:snvbConstEqLenFinMap} for $\PP^{1 \times 3}$ 
and $\PP^{1 \times 4}$, respectively.
Define a closed immersion 
$\alpha\colon\PP^{1\times 3}\rightarrow \PP^{1\times 4}$ 
by the vanishing of the coordinate $y_1^{(4)}$. 
Fix coordinates $x_0, \ldots, x_4$ for $\PP^4$, and let 
$\beta\colon \PP^3 \rightarrow \PP^4$ be the closed immersion given 
by the vanishing of $x_4$. 
We now have that the diagram in 
\eqref{enum:snvbConstNonEqLenFinMap} is indeed commutative.

In \eqref{enum:snvbConstNonEqLenPushFwd}, 
we take $\mathcal L =\cO_{\mathbb P^{1\times 3}}(1,2,3)$ and 
$\mathcal L' = \cO_{\PP^{1\times 4}}(0,1,2,3)$ and set 
$\EE_f = \pi^{(3)}_* \cL$ and $\EE_{f'} = \pi^{(4)}_* \cL'$. 
Set $\EE:=\beta_* \EE_f$ and $\EE':=\EE_{f'}$. 
Then $\EE$ is a supernatural sheaf on $\PP^4$ 
(see Remark~\ref{remark:pushFwdSupNat}), and 
\begin{align*}
\Hom_{\PP^{4}}(\EE', \EE) 
 &= \Hs^0\left(\PP^{4}, \sheafHom\left( \pi^{(4)}_*  \left( \cO_{\PP^{1\times
 4}}(0,1,2,3)\right), \pi^{(4)}_*\left( \alpha_*\cO_{\PP^{1\times 3}}(1,2,3)
 \right) \right) \right).&
\intertext{By Remarks~\ref{rmk:inclusion} and~\ref{rmk:0s}, we obtain the containment}
\Hom_{\PP^{4}}(\EE', \EE)&\supseteq \Hs^0\left(\PP^{4}, \pi^{(4)}_* \sheafHom\left( \cO_{\PP^{1\times 4}}(0,1,2,3), \alpha_*\cO_{\mathbb P^{1\times 3}}(1,2,3) \right)\right) \\
&\cong \Hs^0\left(\PP^{1\times 4},  \sheafHom\left(  \cO_{\mathbb
      P^{1\times 4}}(0,1,2,3),\alpha_*\cO_{\mathbb P^{1\times
        3}}(1,2,3) \right) \right)\\
&\cong \Hs^0\left(\PP^{1\times 4}, \left( \alpha_*\cO_{\mathbb P^{1\times 3}}(1,1,1)\right)(0,0,0,-3) \right)&\\
&\cong \Hs^0\left(\PP^{1\times 4}, \alpha_*\cO_{\mathbb P^{1\times 3}}(1,1,1)\right)\cong \Bbbk^{8}.
\end{align*}
In particular, $\Hom_{\PP^{4}}(\EE', \EE)\ne 0$, as desired.
\end{example}

\begin{remark}\label{rmk:0s}
Let $1 \leq s < t$, and let $\alpha\colon\PP^{1\times s} \rightarrow \PP^{1
\times t}$ be the embedding given by the vanishing of $y^{(s+1)}_1, \ldots,
y^{(t)}_1$. 
Let $\mathcal F$ be a coherent sheaf on $\PP^{1 \times s}$ and 
$b\in \ZZ^{t-s}$. Write $\mathbf{0}_s$ for the
$0$-vector in $\ZZ^s$. Then
\begin{equation}
\Hs^i\left( \PP^{1\times t}, \left( \alpha_* \FF\right) (\mathbf{0}_s,b)
\right) 
\cong 
\Hs^i\left( \PP^{1\times t}, \alpha_* \FF\right)
\cong 
\Hs^i\left( \PP^{1\times s}, \FF\right)
\end{equation}
The first isomorphism
follows from the projection formula, taken along with 
the fact that, by the definition of $\alpha$, the line bundle
$\cO_{\PP^{1\times t}}(\mathbf{0}_s,b)$ is trivial when restricted to the
support of $\alpha_* \FF$ (which is contained in $\PP^{1 \times s}$). 
The second isomorphism holds because $\alpha$ is a finite morphism.
\end{remark}

\begin{prop}
\label{propn:rootDiffLength}
If $\ell(f) < \ell(f')$, then Theorem~\ref{thm:main:root} holds.
\end{prop}

\begin{proof}
We may reduce to the case $\ell(f')=n-1$ by the same argument as in the beginning of the proof of Proposition~\ref{propn:rootSameLength}.

Let $s = \ell(f)$ and consider the line bundles 
$\mathcal L = \cO_{\PP^{1 \times s}}(-\tau(f) - \mathbf 1)$ 
on $\PP^{1 \times s}$ and 
$\mathcal L' = \cO_{\PP^{1 \times(n-1)}}(-f' - \mathbf 1)$ 
on $\PP^{1 \times (n-1)}$. 
Let $\pi\colon\PP^{1\times s} \to \PP^s$ and 
$\pi'\colon \PP^{1 \times (n-1)} \to \PP^{n-1}$ be the maps 
defined by the forms in~\eqref{eqn:fp}. 
Let $\EE_f = \pi_* \mathcal L$ and 
$\EE_{f'} = (\pi')_* \mathcal L'$, and define the closed immersion 
$\alpha\colon\PP^{1\times s}\rightarrow \PP^{1\times(n-1)}$ 
by the vanishing of the coordinates $y_1^{(s+1)}, \ldots,y_1^{(n-1)}$. 
Fix coordinates $x_0, \ldots, x_{n-1}$ for $\PP^{n-1}$, and
let $\beta\colon \PP^s \rightarrow \PP^{n-1}$ be the closed immersion 
given by the vanishing of $x_{s+1}, \ldots, x_{n-1}$. 
This yields the commutative diagram
\[
\xymatrix{
\PP^{1\times s} \ar[r]^-{\alpha} \ar[d]^{\pi} & 
\PP^{1 \times (n-1)} \ar[d]^{\pi'} \\
\PP^s \ar[r]^{\beta} & \PP^{n-1}.
}
\]
By Remark~\ref{remark:pushFwdSupNat}, 
$\EE := \beta_* \EE_f$ is a supernatural sheaf of type $f$. 
Also, $\EE' := \EE_{f'}$ is a supernatural sheaf of type $f'$. 

We must show that $\Hom_{\PP^{n-1}}(\EE', \EE) \neq 0$. 
It suffices to show that 
$\Hom_{\PP^{1 \times
(n-1)}}(\mathcal L', \alpha_*\mathcal L) \neq 0$ 
by Remark~\ref{rmk:inclusion}. 
To see this, let $c:=(f_1', \dots, f'_s)$ and $b:=(-f'_{s+1}-1, \dots, -f'_{s'}-1)$, 
and note that 
\begin{align*}
\sheafHom (\mathcal L', \alpha_*\mathcal L) & = 
\sheafHom(\cO_{\PP^{1\times (n-1)}}(-f'-\mathbf 1), 
\alpha_* \cO_{\PP^{1\times s}}(-\tau(f)-\mathbf 1)) \\ 
& \cong (\alpha_* \cO_{\PP^{1\times s}}(c-\tau(f)))(\mathbf 0_s, -b).
\end{align*}
By Remark~\ref{rmk:0s}, 
$\Hom (\mathcal L', \alpha_*\mathcal L) = \Hs^0(\PP^{1 \times s},
\cO(c-\tau(f)))$, which is nonzero as $\tau(f) \preceq c$.
\end{proof}

\section{Equivariant construction of morphisms between supernatural sheaves}%
\label{sec:equiv:root}

Throughout this section, we assume that $\Bbbk$ is a field of
characteristic 0 and that all root sequences have length $n-1$.  Let $V$ be
an $n$-dimensional $\Bbbk$-vector space, identify $\PP^{n-1}$ with
$\PP(V)$, and let $\cQ$ denote the tautological quotient bundle of
rank $n-1$ on $\PP(V)$.  We have a short exact sequence
\[
0 \to \cO(-1) \to V \otimes \cO_{\PP(V)} \to \cQ \to 0.
\]
We will use the fact that $\det \cQ \cong \cO(1) \otimes \bigwedge^n
V$ is a $\GL(V)$-equivariant isomorphism. For a weakly decreasing
sequence $\lambda$ of non-negative integers, we let $\Sc_\lambda$
denote the corresponding Schur functor. See \cite{weyman}*{Chapter 2}
for more details (since we are working in characteristic 0, the
functors $K_\lambda$ and $L_{\lambda^t}$ are isomorphic, where
$\lambda^t$ is the transpose partition of $\lambda$, and we call this
$\Sc_\lambda$). We extend this definition to weakly decreasing
sequences $\lambda$ with possibly negative entries as follows. Set
$\mathbf{1}=(1,\dots, 1)\in \mathbb Z^{n-1}$ and define
$\Sc_{\lambda}\cQ:=\Sc_{\lambda - \lambda_{n-1}\mathbf{1}}\cQ \otimes
(\det \cQ)^{\lambda_{n-1}}$.

\begin{proof}[Proof of Theorem~\ref{thm:equivariant:root}] 
  The reverse implication has been shown in
  Proposition~\ref{prop:half:poset:root}.  For the forward
  implication, we proceed in two steps.  First, we construct
  equivariant supernatural bundles $\EE'$ and $\EE$ with $\Hom(\EE',
  \EE) \ne 0$ using the construction in the proof of \cite[Theorem
  6.2]{EiScConjOfBS07}. Second, we use this fact to construct a new
  supernatural bundle $\EE''$ of type $f'$ such that
  $\Hom_{\GL(V)}(\EE'',\EE)\ne 0$.  Thus we will ignore powers of the
  trivial bundle $\bigwedge^n V$ that appear in the first step.
    
Write $N_i = f'_i - f_i$ and let $\lambda\in \mathbb Z^{n-1}$ be the 
partition defined by
\[
\lambda_i:= f_1 - f_{n-i} - n+1 + i \quad \text{for } 1 \le i \le n-1.
\]
Let $\lambda'$ be the sequence of weakly decreasing integers defined by 
$\lambda'_{n-i} := \lambda_{n-i} - N_i$ and set
\[
\EE := \Sc_\lambda \cQ \otimes \cO(-f_1 - 1)\quad \text{ and } \quad
\EE' := \Sc_{\lambda'} \cQ \otimes \cO(-f_1 - 1).
\]
Observe that $\Sc_{\lambda'} \cQ \otimes \cO(-f_1 - 1) \cong
\Sc_{\lambda' + N_1\cdot \mathbf{1}} \cQ \otimes \cO(-f'_1 - 1)$.
Hence by the Borel--Weil--Bott theorem~\cite{weyman}*{Corollary
  4.1.9}, $\EE$ and $\EE'$ are supernatural vector bundles of types
$f$ and $f'$, respectively.
  
To compute $\Hom(\EE',\EE)$, let $\lambda'' := \lambda' + N_1\cdot
\mathbf{1}$. Define $\lambda^c$ to be the complement of $\lambda$
inside of the $(n-1) \times \lambda_1$ rectangle, so $\lambda^c_j =
\lambda_1 - \lambda_{n-j}$ for $1 \le j \le n-1$.  Then
$\Sc_{\lambda}\cQ\cong \Sc_{\lambda^c}\cQ^* \otimes \cO(\lambda_1)$ by
\cite{weyman}*{Exercise 2.18}.  We then obtain
\[
\sheafHom(\EE', \EE) \cong \Sc_{\lambda'} \cQ^* \otimes \Sc_\lambda
\cQ \cong \Sc_{\lambda''} \cQ^* \otimes \Sc_{\lambda^c} \cQ^* \otimes
\cO(\lambda_1 + N_1)
\]
and seek to show that this bundle has a nonzero global section.
  
Fix $\mu$ so that $\Sc_\mu \cQ^*$ is a direct summand of
$\Sc_{\lambda''} \cQ^* \otimes \Sc_{\lambda^c} \cQ^*$.  The
Borel--Weil--Bott Theorem~\cite{weyman}*{Corollary 4.1.9} shows that
$\Sc_\mu \cQ^* \otimes \cO(\lambda_1 + N_1)$ has nonzero sections if and
only if $\lambda_1 + N_1 \ge \mu_1$.  This is equivalent to $\mu$ being
inside of a $(n-1) \times (\lambda_1+N_1)$
rectangle. By~\cite{fulton}*{\S9.4}, the existence of such a $\mu$ is
equivalent to the condition
\begin{equation}\label{eqn:lambda:ineq}
  \lambda''_i + \lambda^c_{n-i} \le  \lambda_1+N_1 \quad \text{for }
  i=1,\dots,n-1.
\end{equation} 
Since $\lambda''_i + \lambda^c_{n-i} = \lambda_1 + N_1 - N_{n-i}$, we
see that \eqref{eqn:lambda:ineq} holds for all $i$, and thus
$\Hom(\EE',\EE)\ne 0$.

For the second step, replace $\EE'$ by $\EE'' : = \EE' \otimes
\Hom(\EE', \EE)$, where we view $\Hom(\EE', \EE)$ as a trivial bundle
over $\PP(V)$.  Note that
\[
\Hs^i(\PP(V),\EE''(j))\cong \Hs^i(\PP(V),\EE'(j))\otimes \Hom(\EE',
\EE)
\]
for all $i,j$, and hence $\EE''$ is also supernatural of type $f'$.
The space of sections $\Hom(\EE'', \EE)$ is $\Hom(\EE', \EE)^* \otimes
\Hom(\EE', \EE)$, which contains the $\GL(V)$-invariant section
corresponding to the evaluation map. This gives a nonzero
$\GL(V)$-equivariant map $\EE'' \to \EE$.
\end{proof}

\begin{example}
  We reconsider Example~\ref{ex:1} in the equivariant context. Here we
  will not ignore powers of $\bigwedge^n V$.  Let $n=4$ and
  $f=(-2,-3,-4,-5) \preceq f'=(-1,-2,-3,-4)$.  With notation as in the
  proof of Theorem~\ref{thm:equivariant:root}, we have $N=(1,1,1,1)$,
  $\lambda=(0,0,0,0)$, $\lambda'=(-1,-1,-1,-1)$,
  \begin{align*}
    \EE&=\Sc_{(0,0,0,0)}\cQ \otimes \cO(2-1)=\cO(1), \quad \text{and}\\
    \EE'&=\Sc_{(-1,-1,-1,-1)}\cQ \otimes \cO(2-1) = \left(\cO(-1)
      \otimes \left(\bigwedge^n V\right)^{-1} \right) \otimes \cO(1)
      =\left(\bigwedge^n V\right)^{-1} \otimes \cO.
  \end{align*}
  Since $\lambda^c=(0,0,0,0)=\lambda''$, we see that
  \[
  \sheafHom(\EE',\EE)\cong \cO(1) \otimes \bigwedge^n V,
  \]
  which certainly has nonzero global sections.  In fact,
  $\Hom(\EE',\EE)\cong V \otimes \bigwedge^n V$.  Note, however, that
  this implies that there is no nonzero equivariant morphism from
  $\EE'$ to $\EE$. We thus set $\EE'':=\EE\otimes \Hom(\EE',\EE)$.
  Then $\Hom(\EE'',\EE)\cong V^*\otimes V$, and our desired nonzero
  equivariant morphism is given by the trace element.
\end{example}

\section{Remarks on other graded rings}
\label{sec:extensions}

Given any graded ring $R$, one could try to use an analog of
Theorem~\ref{thm:poset:deg:main} to induce a partial order on
the extremal rays of the cone of Betti diagrams over $R$.  This
application has already proven useful in a couple of the other cases
where Boij--S\"oderberg has been studied.  In this section, we provide
a sketch of some of these applications.

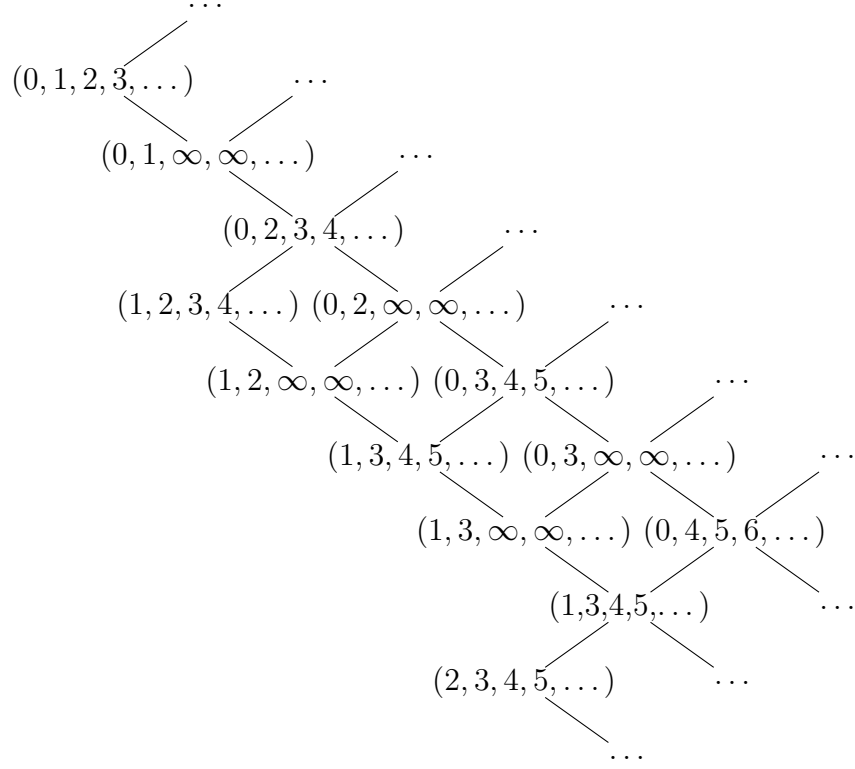
\begin{figure}
\begin{tikzpicture}[xscale=1.4,yscale=1.0]
\draw(0,9) node {$(0,1,2,3,\dots)$};
\draw(1,10) node {\dots};
\draw(1,8) node {$(0,1,\infty,\infty,\dots)$};
\draw(2,9) node {\dots};
\draw(3,8) node {\dots};
\draw(2,7) node {$(0,2,3,4,\dots)$};
\draw(3,6) node {$(0,2,\infty,\infty,\dots)$};
\draw(1,6) node {$(1,2,3,4,\dots)$};
\draw(4,7) node {\dots};
\draw(2,5) node {$(1,2,\infty,\infty,\dots)$};
\draw(3,4) node {$(1,3,4,5,\dots)$};
\draw(4,5) node {$(0,3,4,5,\dots)$};
\draw(5,6) node {\dots};
\draw(4,3) node {$(1,3,\infty,\infty,\dots)$};
\draw(5,4) node {$(0,3,\infty,\infty,\dots)$};
\draw(6,5) node {\dots};
\draw(4,1) node {$(2,3,4,5,\dots)$};
\draw(5,2) node {(1,3,4,5,\dots)};
\draw(6,3) node {$(0,4,5,6,\dots)$};
\draw(7,4) node {\dots};
\draw(5,0) node {\dots};
\draw(6,1) node {\dots};
\draw(7,2) node {\dots};
\draw[-] (1.2,6.2)--(1.8,6.8);
\draw[-] (1.2,5.8)--(1.8,5.2);
\draw[-] (0.2,9.2)--(.8,9.8);
\draw[-] (1.2,8.2)--(1.8,8.8);
\draw[-] (2.2,7.2)--(2.8,7.8);
\draw[-] (3.2,6.2)--(3.8,6.8);
\draw[-] (4.2,5.2)--(4.8,5.8);
\draw[-] (5.2,4.2)--(5.8,4.8);
\draw[-] (6.2,3.2)--(6.8,3.8);
\draw[-] (0.2,8.8)--(0.8,8.2);
\draw[-] (1.2,7.8)--(1.8,7.2);
\draw[-] (2.2,6.8)--(2.8,6.2);
\draw[-] (3.2,5.8)--(3.8,5.2);
\draw[-] (4.2,4.8)--(4.8,4.2);
\draw[-] (5.2,3.8)--(5.8,3.2);
\draw[-] (2.2,5.2)--(2.8,5.8);
\draw[-] (3.2,4.2)--(3.8,4.8);
\draw[-] (4.2,3.2)--(4.8,3.8);
\draw[-] (5.2,2.2)--(5.8,2.8);
\draw[-] (2.2,4.8)--(2.8,4.2);
\draw[-] (3.2,3.8)--(3.8,3.2);
\draw[-] (4.2,2.8)--(4.8,2.2);
\draw[-] (4.2,1.2)--(4.8,1.8);
\draw[-] (4.2,0.8)--(4.8,0.2);
\draw[-] (5.2,1.8)--(5.8,1.2);
\draw[-] (6.2,2.8)--(6.8,2.2);
\end{tikzpicture}
\caption{For 
  the hypersurface ring $R$, this partial order provides a simplicial fan structure, 
  as illustrated in~\cite{bbeg} and discussed in Example~\ref{ex:hypersurface}.  
  The partial order is determined by an analog of 
  Theorem~\ref{thm:poset:deg:main}.}
\label{fig:hypersurface}
\end{figure}

\begin{example}\label{ex:hypersurface}
  We first consider an example involving hypersurface rings over
  $\Bbbk[x,y]$. Let $f\in \Bbbk[x,y]$ be a quadric polynomial, and set
  $R:=\Bbbk[x,y]/\<f\>$.  The cone of Betti diagrams over $R$ is
  described in detail in \cite{bbeg}.  The extremal rays still
  correspond to Cohen--Macaulay modules with pure resolutions, though
  some of the degrees are infinite in length.
  
  \begin{enumerate}[(i)]
  \item \emph{Finite pure resolutions.}  For example, if $h$ is a
    degree $7$ polynomial that is not divisible by $f$, then the free
    resolution of $R/\<h\>$ is \[ R\leftarrow R(-7) \leftarrow 0.\]
    Following the notation of Section~\ref{sec:prelim deg}, we denote
    such a resolution by its corresponding degree sequence, i.e.,
    $(0,7,\infty,\infty,\dots)$.
\item 
  \emph{Infinite pure resolutions.}  
  For example, the free resolution of the $R$-module $R/\<x,y\>$ is
  \[
  R\leftarrow R^2(-1)\leftarrow R^2(-2)\leftarrow R^2(-3)\leftarrow
  \cdots.
  \]
  We denote this by its corresponding degree sequence, i.e.,
  $(0,1,2,3,\dots)$.
\end{enumerate}

There are two possible partial orders for these extremal rays:
\begin{itemize}
\item $\rho_d\preceq \rho_{d'}$ if $d_i\leq d_i'$ for all $i$.
\item $\rho_d\preceq\rho_{d'}$ if there exist Cohen--Macaulay
  $R$-modules $M$ and $M'$ with pure resolutions of types $d$ and
  $d'$, respectively, with $\Hom_R(M',M)_{\leq 0}\ne 0$.
\end{itemize}
In contrast with the case of the polynomial ring, these partial orders are genuinely different.  Only the second partial order leads to a greedy algorithm for decomposing Betti diagrams over $R$, in parallel to~\cite{EiScConjOfBS07}*{Decomposition Algorithm}.  This also provides an analog of the Multiplicity Conjecture for $R$.
\end{example}

\begin{example}\label{ex:bigraded}
  We now consider $S=\Bbbk[x,y]$ with the $\ZZ^2$-grading $\deg(x):=(1,0)$
  and $\deg(y):=(0,1)$.  In general, the cone of bigraded Betti diagrams 
  over $S$ remains poorly understood.  However, 
  portions of this cone have been worked out by the first
  three authors, and we now provide a brief sketch of these unpublished results.

  We restrict attention to the cone of Betti diagrams of finite length
  $S$-modules $M$, where all of the Betti numbers of $M$ are
  concentrated in bidegrees $(a,b)$ with $0\leq a,b\leq 2$.  The
  extremal rays of this cone may be realized by quotients of monomial
  ideals of the form $m_1/m_2$, where each $m_i$ is a monomial ideal
  generated by monomials of the form $x^\ell y^k$ with $0\leq
  \ell,k\leq 2$.  
  The natural analog of Theorem~\ref{thm:poset:deg:main} induces 
  a partial order on these rays, which also induces a simplicial structure 
  on this cone of bigraded Betti diagrams.
\end{example}
\def\cfudot#1{\ifmmode\setbox7\hbox{$\accent"5E#1$}\else
  \setbox7\hbox{\accent"5E#1}\penalty 10000\relax\fi\raise 1\ht7
  \hbox{\raise.1ex\hbox to 1\wd7{\hss.\hss}}\penalty 10000 \hskip-1\wd7\penalty
  10000\box7}

\bigskip

\end{document}